\documentclass[final,onefignum,onetabnum]{siamonline171218}

\usepackage{braket,amsfonts}
\usepackage{soul}
\usepackage{multicol}
\usepackage{array}
\usepackage{tikz}
\usepackage{float}
\usepackage{enumerate}
\usepackage{caption}
\usepackage{subcaption}
\usepackage{pgfplots}

\newsiamthm{claim}{Claim}
\newsiamremark{remark}{Remark}
\newsiamthm{Example}{Example}
\newsiamremark{hypothesis}{Hypothesis}
\crefname{hypothesis}{Hypothesis}{Hypotheses}

\def\cred{\color{black}}
\newcommand{\aaa}[1]{{\color{blue}#1}}
\newcommand{\aaar}[1]{{\color{purple}#1}}

\newcommand{\ignore}[1]{}
\DeclareMathOperator{\conv}{conv}

\renewcommand{\aaa}[1]{#1}\renewcommand{\st}[1]{}
\renewcommand{\aaar}[1]{{#1}}

\newcommand {\iref}[1]{(\ref{#1})}
\newcommand {\beq}{\[}\newcommand {\eeq}{\]}
\newcommand {\beqn}{\begin{equation}}\newcommand {\eeqn}{\end{equation}}
\newcommand {\beqan}{\begin{eqnarray}}\newcommand {\eeqan}{\end{eqnarray}}
\newcommand {\beqa}{\begin{eqnarray*}}\newcommand {\eeqa}{\end{eqnarray*}}
\newcommand{\lc}{\left\lceil}\newcommand{\rc}{\right\rceil}

\def\skipit#1{}

\newcommand {\R}{\mathbb{R}}

\def\skipit#1{}

\usepackage{algorithmic}

\usepackage{graphicx,epstopdf}

\Crefname{ALC@unique}{Line}{Lines}

\usepackage{amsopn}

\usepackage{xspace}
\usepackage{bold-extra}
\usepackage[most]{tcolorbox}

\colorlet{texcscolor}{blue!50!black}
\colorlet{texemcolor}{red!70!black}
\colorlet{texpreamble}{red!70!black}
\colorlet{codebackground}{black!25!white!25}

\lstdefinestyle{siamlatex}{%
  style=tcblatex,
  texcsstyle=*\color{texcscolor},
  texcsstyle=[2]\color{texemcolor},
  keywordstyle=[2]\color{texemcolor},
  moretexcs={cref,Cref,maketitle,mathcal,text,headers,email,url},
}

\tcbset{%
  colframe=black!75!white!75,
  coltitle=white,
  colback=codebackground, 
  colbacklower=white, 
  fonttitle=\bfseries,
  arc=0pt,outer arc=0pt,
  top=1pt,bottom=1pt,left=1mm,right=1mm,middle=1mm,boxsep=1mm,
  leftrule=0.3mm,rightrule=0.3mm,toprule=0.3mm,bottomrule=0.3mm,
  listing options={style=siamlatex}
}

\newtcblisting[use counter=example]{example}[2][]{%
  title={Example~\thetcbcounter: #2},#1}

\newtcbinputlisting[use counter=example]{\examplefile}[3][]{%
  title={Example~\thetcbcounter: #2},listing file={#3},#1}

\DeclareTotalTCBox{\code}{ v O{} }
{ 
  fontupper=\ttfamily\color{black},
  nobeforeafter,
  tcbox raise base,
  colback=codebackground,colframe=white,
  top=0pt,bottom=0pt,left=0mm,right=0mm,
  leftrule=0pt,rightrule=0pt,toprule=0mm,bottomrule=0mm,
  boxsep=0.5mm,
  #2}{#1}

\patchcmd\newpage{\vfil}{}{}{}
\begin{tcbverbatimwrite}{tmp_\jobname_header.tex}
\title{Robust-to-Dynamics Optimization   
\thanks{This work was partially funded by the DARPA Young Faculty Award, the Young Investigator Award of the AFOSR, the CAREER Award of the NSF, the Google Faculty Award, and the Sloan Fellowship. The authors would also like to thank the Mathematical Institute at Oxford University for hosting them while part of this research was being conducted.}}

\author{Amir Ali Ahmadi \thanks{ORFE, Princeton University, Sherrerd Hall, Princeton, NJ 08540 (\email{aaa@princeton.edu})}   \and
	Oktay G\"unl\"uk \thanks{Cornell University, School of ORIE, Ithaca, NY 14853 (\email{ong5@cornell.edu}). The author was partially supported by ONR grant N00014-21-1-2575}
}

\headers{Robust-to-Dynamics  Optimization}{A. A. Ahmadi and O. G\"unl\"uk}
\end{tcbverbatimwrite}
\input{tmp_\jobname_header.tex}

\ifpdf
\hypersetup{ pdftitle={Robust-to-Dynamics Optimization} }
\fi

\begin{document}
\maketitle

\begin{tcbverbatimwrite}{tmp_\jobname_abstract.tex}
\begin{abstract}
A \emph{robust-to-dynamics optimization (RDO) problem} is an optimization problem specified by two pieces of input: (i) a mathematical program (an objective function $f:\mathbb{R}^n\rightarrow\mathbb{R}$ and a feasible set $\Omega\subseteq\mathbb{R}^n$), and (ii) a dynamical system (a map $g:\mathbb{R}^n\rightarrow\mathbb{R}^n$). Its goal is to minimize $f$ over the set $\mathcal{S}\subseteq\Omega$ of initial conditions that forever remain in $\Omega$ under $g$. \aaa{The focus of this paper is on the case where the mathematical program is a linear program and the dynamical system is either a known linear map, or an uncertain linear map that can change over time.} In both cases, we study a converging sequence of polyhedral outer approximations and (lifted) spectrahedral inner approximations to $\mathcal{S}$. Our inner approximations are optimized with respect to the objective function $f$ and their semidefinite characterization---which has a semidefinite constraint of fixed size---is obtained by applying polar duality to convex sets that are invariant under (multiple) linear maps. We characterize three barriers that can stop convergence of the outer approximations to $\mathcal{S}$ from being finite. We prove that once these barriers are removed, our inner and outer approximating procedures find an optimal solution and a certificate of optimality for the RDO problem in a finite number of steps. Moreover, in the case where the dynamics are linear, we show that this phenomenon occurs in a number of steps that can be computed in time polynomial in the bit size of the input data. Our analysis also leads to a polynomial-time algorithm for RDO instances where the spectral radius of the linear map is bounded above by any constant less than one. Finally, in our concluding section, we propose a broader research agenda for studying \emph{optimization problems with dynamical systems constraints}, of which RDO is a special case.
\end{abstract}


\end{tcbverbatimwrite}
\input{tmp_\jobname_abstract.tex}

\section{Introduction}

In many real-world situations, a decision maker is called upon to make a decision that
optimizes an objective function and satisfies a current set of constraints. These constraints however can change over time because of the influence of an external dynamical system, rendering the original decision infeasible. \aaa{For example, a vaccine that meets a certain level of efficacy against a given virus can fail to do so under future mutations of the virus.} The goal is then to make the best decision among all possible decisions which remain feasible as the constraints vary over time. By changing the point of reference, such a scenario can equivalently be thought of as a situation where a constraint set is fixed, but the original decision is moved around under the influence of a dynamical system. The goal is then to make the best decision among all decisions that remain in the (fixed) constraint set under the influence of the dynamical system.



\aaa{\st{As an illustration of problems of this type, consider the case of wild animal culling, which is the practice of artificially modifying the size of animal populations via hunting. In the United States e.g., deer is culled to limit the propagation of Lyme disease of which they are a vector. In this scenario, experts would like to minimize the number of animals that are allowed to be hunted. The constraint here is to keep their future population, which is determined by the population dynamics of the animal, under a certain level (e.g., to impede propagation of the disease they carry), but above some other level (e.g., to ensure that the species will not go extinct). Problems that can be modeled in this way also occur in infrastructure or medical equipment design. For example, one may want to find the minimum amount of tension that should be applied to the cables in a suspension bridge so that it remains stable under the dynamics that are applied to it. These can be the result of, e.g., pedestrian or vehicle movement on the bridge, or natural disasters like earthquakes. Similarly, in safety-critical medicine, one may want to decide on the minimum width of a stent that goes into an artery during an angioplasty procedure so that blood flow through the stent remains above a certain level and the stent stays in place. These constraints should be met at all times even though the width of the stent can vary over time as the artery undergoes changes depending on the body's temperature, and the viscosity and pressure of the blood.}}

%
%
%
%

In this paper, we study a mathematical abstraction of problems of this nature. We will refer to them as \emph{``robust-to-dynamics optimization'' (RDO)} problems. The name alludes to the fact that the solution to these problems needs to be robust to external dynamics, in the sense that it should remain feasible at all times as it is moved by a dynamical system.
%
%
RDO problems have a very natural mathematical formulation which relies on two pieces of input:

\begin{enumerate}
\item an optimization problem:
\begin{equation}\label{eq:opt.input}
\min_{x} \{f(x): x\in\Omega\},
\end{equation}

\item a dynamical system:
\aaa{\begin{equation}\label{eq:dynamics.input}
\begin{cases}
 x_{k+1}=g(x_k)& \  \mbox{in discrete time (focus of this paper), or}\\
~~~~ \dot{x}=g(x)& \ \mbox{in continuous time}.
\end{cases}
\end{equation}
}
\end{enumerate}
Here, we have $x\in\mathbb{R}^n$, $\Omega\subseteq\mathbb{R}^n$, $f:\mathbb{R}^n\rightarrow\mathbb{R}$, $g:\mathbb{R}^n\rightarrow\mathbb{R}^n$; $x_k$ denotes the state at time step $k$, and $\dot{x}$ is the derivative of $x$ with respect to time. RDO is then the following optimization problem:\footnote{ \aaa{We write ``u.t.d.'' as an abbreviation for ``under the dynamics''.}}
\begin{equation}\label{eq:rdo.DT}
\min_{x_0}\{f(x_0): x_k\in\Omega \text{ for } k=0,1,2,\ldots, \text{ u.t.d. } x_{k+1}=g(x_k)  \}
\end{equation}
in discrete time, or

\begin{equation}\nonumber
\min_{x_0}\{f(x_0): x(t;x_0)\in\Omega, \forall t\geq 0, \text{ u.t.d. } \dot{x}=g(x)  \}
\end{equation}
in continuous time, where $x(t;x_0)$ denotes the solution of the differential equation $\dot{x}=g(x)$ at time $t$, starting at the initial condition $x_0\in\mathbb{R}^n$. \aaa{By the u.t.d. notation, we imply that} $x_k$ (resp. $x(t;x_0)$) must satisfy the equation of our dynamics $x_{k+1}=g(x_k)$ (resp. $\dot{x}=g(x)$) for $k=0,1,2,\ldots$ (resp. $t\geq 0$). In words, we are optimizing an objective function $f$ over the set of initial conditions that never leave the set $\Omega$ under the dynamics governed by $g$.

RDO problems can naturally be categorized depending on the type of optimization problem considered in (\ref{eq:opt.input}) and the type of dynamics considered in (\ref{eq:dynamics.input}). A list of some possible combinations is given in the table below.

\begin{table}[H]
	\centering
	\begin{tabular}{l|l}
		\textbf{Optimization Problem $(f,\Omega)$} & \textbf{Dynamical System $(g)$}                 \\ \hline
		Linear program                & Linear                            \\
		Convex quadratic program            & Nonlinear (e.g., polynomial)                         \\
		Second order cone program           & Uncertain                         \\
		Semidefinite program           & Time-varying                      \\
			Integer program         & Uncertain and time-varying \\
		\vdots         & \vdots 
	\end{tabular}
		\caption{		Any combination of an entry from the first column and an entry from the second column of this table leads to an RDO problem.}
		\label{tab:rdo}
\end{table}


\aaa{This framework can be even further generalized by considering optimization problems with ``dynamical system constraints'' (see Section~\ref{sec:opt.with.DS}). In this paper, we focus on RDO problems where the mathematical program is a linear program and the dynamical system is discrete and either a known linear map, or an uncertain linear map that can change over time. The second case subsumes some uncertain nonlinear dynamics that are not time-varying as well.}

As is made evident above, RDO is at the juncture between optimization and dynamical systems, and as such, has some commonalities with literature in both areas. On the optimization side, RDO comes closest to the area of robust optimization and on the dynamical systems side, it has connections to the theory of invariant sets. We describe the links between these areas in more detail below.

In its most common form, robust optimization  (RO)~\cite{robust_opt_book,bertsimas_ro_survey,mulvey1995robust,bertsimas2004robust,ben2002robust} deals with problems of the type
\begin{equation}
\min_{x} \{f(x):u_i(x)\leq 0\  \forall u_i \in\mathcal{U}_i, i=1,\ldots, m\},
\end{equation}
where $\mathcal{U}_i$ is a prescribed uncertainty set for (the parameters of) the constraint function $u_i:\mathbb{R}^n\rightarrow\mathbb{R}$. Like RDO, RO problems have to contend with uncertainty in the constraints, though unlike RDO, this uncertainty is not explicitly linked to a dynamical system. As an area, RO is well-studied from a computational complexity standpoint. By now, we almost fully understand when a robust optimization problem involving a particular mathematical program  (e.g., a linear or a convex quadratic program) and a particular type of uncertainty set $\mathcal{U}_i$ (e.g., polyhedral, ellipsoidal, etc.) is polynomial-time solvable or NP-hard; see~\cite{bertsimas_ro_survey} and~\cite{robust_opt_book} and references therein. \aaa{We note that similar to robust optimization problems, the RDO framework is a special case of semi-infinite optimization \cite{LOPEZ2007491,DJELASSI2021100006,ReemtsenRuckmannBook}.
In the case of RDO problems with discrete deterministic dynamics, one is dealing with a countable semi-infinite optimization problem. These constraints, however, have a very special structure arising from the underlying dynamical system, and we exploit this structure heavily in this paper.}





On the dynamics side, invariant set theory~\cite{blanchini1999set,Khalil:3rd.Ed,bitsoris1988positively,gilbert1991linear,mezic1999method,korda2014convex,athanasopoulos2016computing} concerns itself with the study of sets that are invariant under the action of dynamical systems. It also considers the problem of designing controllers that influence a dynamical system so as to make a desired set invariant. This problem has applications to model predictive control~\cite{kerrigan2000invariant,garcia1989model,legat2018computing}, among other subfields in control theory. The literature on invariance in control by and large does not consider constraints and studies the existence and the structure of invariant sets for different types of dynamical systems. The subset of the literature that does incorporate constraints typically aims at characterizing the maximal (with respect to inclusion) invariant set within a given constraint set. These maximal invariant sets are often complicated to describe even for simple dynamical systems and hence are approximated. While both inner and outer approximations are interesting, inner approximations are more relevant to applications as they provide initial conditions that remain within the constraint set for all time. To the best of our knowledge, the approximations available in the literature do not take into consideration an objective function as is done in this paper (cf. Section~\ref{subsubsec:invariant.ellipsoid.SDP} and Section~\ref{subsubsec:IrEr.switched}).

While problems related to RDO have been studied in the control community, we believe that the framework in which we present these problems---as a merging of a mathematical program and a dynamical system (cf. Table~\ref{tab:rdo})---and the questions that we study, and hence our results, are different. Our hope is that this framework will make problems of this type more palatable to the optimization community. We also believe that the framework provides the right setup for a systematic algorithmic study of RDO, as has been done so successfully for RO. Indeed, our overarching goal is to provide an understanding of the computational complexity of each type of RDO problem that arises from the two columns of Table~\ref{tab:rdo}. This can be either in the form of negative results (e.g., NP-hardness or undecidability proofs), or positive ones (algorithms with guarantees). The current paper is a step in this direction.


\aaa{Very recently, the work presented in this paper has been applied to the problem of safely learning dynamical systems from trajectory data~\cite{safe_learning}. This problem has applications to reinforcement learning and robotics and requires the learning process to respect certain ``safety constraints''. These constraints restrict the set of allowable queries (i.e., starting points) to the unknown dynamical system to those whose future trajectories are guaranteed to remain in a given (safe) set; see~\cite{safe_learning} for details. We also believe there are potential applications to constrained nonlinear optimization. Indeed, one can think of any algorithm for minimizing a function $f:\mathbb{R}^n\rightarrow\mathbb{R}$ (e.g., gradient descent) as a dynamical system. If the function is to be minimized over a feasible set $\Omega\subseteq\mathbb{R}^n$, then characterizing the set of initial conditions that remain feasible under the iterations of the algorithm is an RDO problem. In the case where $f$ is a quadratic form, $\Omega$ is polyhedral, and the minimization algorithm is gradient descent (possibly with time-varying step sizes), the resulting RDO problem falls within those studied in this paper.}

\subsection{Outline and contributions of the paper}


\aaa{In Section \ref{sec:R-LD-LP}, we study robust to linear dynamics linear programs (R-LD-LPs), which are RDO problems  where the optimization problem is a linear program and the dynamics are linear.
\st{In Section 2, we study RDO problems in the case where the optimization problem is a linear program and the dynamics are linear: we call these problems R-LD-LPs (for robust to linear dynamics linear programs).We start by studying some basic properties of the feasible set $\mathcal{S}$ of an R-LD-LP.} We show that the feasible set $\mathcal{S}$ of an R-LD-LP is not always polyhedral and that testing membership of a point to $\mathcal{S}$ is NP-hard (Theorem~\ref{thm:S.closed.convex.invariant.non.polyhedral.nphard}).} In Section~\ref{subsec:outer.approx.LDLP}, we study a sequence of natural outer approximations $S_r$ to $\mathcal{S}$ that get tighter and tighter as $r$ grows. We give a polynomial-time checkable criterion for testing whether $S_r=\mathcal{S}$ for a given nonnegative integer $r$ (Lemma~\ref{lemma:termination}). 
We then characterize three conditions that may stop convergence of $S_r$ to $\mathcal{S}$ from happening in a finite number of steps (Propositions~\ref{prop:rho>1.convergence.not.finite},~\ref{prop:convergence.not.finite.without.origin}, ~\ref{prop:P.unbounded.not.finite}, and Theorem~\ref{thm:S.closed.convex.invariant.non.polyhedral.nphard}, part(ii)). These conditions have previously appeared \aaa{as assumptions} in the literature on invariant sets (\aaa{see, e.g.,~\cite[Sect. 2.2]{athanasopoulos2016computing} and references therein}), \aaa{but to} the best of our knowledge, there are no formal arguments that show why all three of the conditions are needed to guarantee finite convergence. Once this is clarified, \aaa{the main theorem of Section~\ref{subsec:outer.approx.LDLP} (Theorem~\ref{thm.polyhedrality}) shows that under these three conditions, convergence of $S_r$ to $\mathcal{S}$ is not only finite but takes a number of steps that can be computed in time polynomial in the size of the data. Our proof also shows that all instances of R-LD-LP for which the spectral radius of the matrix defining the linear dynamics is upper bounded by a constant less than one can be solved in polynomial time.}

 
In Section~\ref{subsec:inner.approx.LDLP}, we study inner approximations of $\mathcal{S}$ which have the advantage (comparatively to outer approximations) of providing feasible solutions to an R-LD-LP. We first give a general construction that starts with any full-dimensional and compact invariant set for the dynamical system and produces a sequence of nested inner approximations to $\mathcal{S}$ that converge to $\mathcal{S}$ finitely (Lemma~\ref{lemma:inner.approx}, Lemma~\ref{lemma:inner.approx.limit}, and Corollary~\ref{corollary:inner.approx.finite}). \aaa{Using this procedure, a finitely-convergent sequence of upper bounds on the optimal value of an R-LD-LP can be computed by solving convex quadratic programs (cf. Section~\ref{subsubsec:IrE.LDLP}). In Section~\ref{subsubsec:invariant.ellipsoid.SDP}, we formulate \aaa{a sequence of semidefinite programs which find inner approximations to $\mathcal{S}$ that are optimally aligned with the objective function of the R-LD-LP~(Theorem~\ref{thm:inner.SDP}, part(i)).} \st{While these problems are nonconvex in their original formulation, we show that they can be reparameterized as a semidefinite program (Theorem~ 2.11, part(i)).} We show that the solutions to our semidefinite programs coincide with the optimal solutions to our R-LD-LP after a number of steps that can be computed in polynomial time (Theorem~\ref{thm:inner.SDP}, part (ii)).}


\aaa{In Section~\ref{sec:R-ULD-LP}, we study robust to \emph{uncertain \aaa{and time-varying}} linear dynamics linear programs (R-UTVLD-LPs), which are RDO problems where the optimization problem is  a linear program and the dynamical system is  linear, but it is uncertain and time-varying. 
\st{Section~3 studies a more intricate class of RDO problems. Here, the optimization problem is still a linear program and the dynamical system is still linear, but it is uncertain and time-varying. We call such problems R-UTVLD-LPs (for robust to uncertain and time-varying linear dynamics linear programs).}
This problem class also captures certain nonlinear time-invariant dynamics in the presence of uncertainty.} The algorithmic questions get considerably more involved here; e.g., even testing membership of a given point to the feasible set of an \aaa{R-UTVLD-LPs} is undecidable. The goal of the section is to generalize (to the extent possible) the results of Section~\ref{sec:R-LD-LP}, both on outer and inner approximations of the feasible set. \aaa{\st{Hence, the structure of Section~3 parallels that of Section~2. We have however chosen to focus in Section~3 mainly on the differences that arise from this more complicated setup.}} \aaa{To do this, we replace} the notion of the spectral radius with that of the joint spectral radius (see Definition~\ref{def:jsr}), and ellipsoidal invariant sets with invariant sets that are a finite intersection of ellipsoids (cf. Theorem~\ref{thm:jsr.max.of.quadratics}). \aaa{\st{With the right generalizations, proofs}} \aaa{We prove finite convergence of both inner and outer approximations and give a semidefinite formulation of inner approximations that are aligned with the objective function (cf. Lemma~\ref{lem:I_rE.properties.switched}, Theorems~\ref{thm:finite.switched} and~\ref{thm:inner.SDP.switched}).} \aaa{\st{The more basic results from Section~2, which can be easily extended to Section~3, have only been proven once---in the R-LD-LP section. This is to make the main ideas of these proofs more accessible to the reader, without the burden of extra notation.}} 
\aaa{Finally, in Section \ref{sec:permutation}, we provide a polynomial-size extended LP formulation of R-UTVLD-LP problems where the dynamical system is described by the set of permutation matrices.}

We conclude \aaa{in Section~\ref{sec:opt.with.DS}} by noting that RDO can be inscribed in a broader class of problems, those that we refer to as ``optimization problems with dynamical system constraints''. \aaa{\st{This framework is described in Section~4.}} We believe that problems of this type form a \aaa{rich} \aaa{\st{very rich and compelling}} class of optimization problems and hope that our paper will instigate further research in this area.

\section{Robust to linear dynamics linear programming}\label{sec:R-LD-LP}
We define a \emph{robust to linear dynamics linear program} (R-LD-LP) to be an optimization problem of the form\footnote{The dependence of the objective function on just the initial condition is without loss of generality. Indeed, if the objective instead read $\sum_{k=0}^N \hat{c}_k^Tx_k,$ we could let $c^T=\sum_{k=0}^N \hat{c}_k^TG^k$ in (\ref{eq:R-LD-LP}) to have an equivalent problem. }
\begin{equation}\label{eq:R-LD-LP}
\min_{x_0\in\R^n}\{c^Tx_0: x_k\in P \ \mbox{for}\ k=0,1,2,\ldots, \mbox{u.t.d.} \ x_{k+1}=Gx_k\},
\end{equation}
where $P\mathrel{\mathop:}=\{x\in\R^n|\ Ax\leq b\}$ is a polyhedron and $G\in\mathbb{R}^{n\times n}$ is a given matrix. One can equivalently formulate an R-LD-LP as a linear program of a particular structure with an infinite number of constraints:
\begin{equation}\label{eq:R-LD-LP2}
\min_{x\in\R^n}\{c^Tx: G^kx\in P \ \mbox{for} \ k=0,1,2,\ldots\}.
\end{equation}
The input to an R-LD-LP is fully defined\footnote{Whenever we study complexity questions around an R-LD-LP, we use the standard Turing model of computation (see, e.g.,~\cite{sipser2006introduction}) and consider instances where the entries of $c,A,b,G$ are rational numbers and hence the input can be represented with a finite number of bits.} by $c\in\R^n,A\in\R^{m\times n},b\in\R^m,G\in\R^{n\times n}.$
%
%
Note that a linear program can be thought of as a special case of an R-LD-LP where the matrix $G$ is the identity matrix. Problem (\ref{eq:R-LD-LP}), or equivalently problem (\ref{eq:R-LD-LP2}),  has a simple geometric interpretation: we are interested in optimizing a linear function not over the entire polyhedron $P$, but over a subset of it that does not leave the polyhedron under the application of $G$, $G^2$, $G^3$, etc. Hence, the feasible set of an R-LD-LP is by definition the following set
\aaa{\begin{equation}\label{eq:feasible.set.R-LD-LP}
\mathcal{S}\mathrel{\mathop:}=\bigcap_{k=0}^{\infty} \{x\in\R^n|\ AG^kx\leq b  \}
=\bigcap_{k=0}^{\infty} \bigcap_{i=1}^{m} \{x\in\R^n|\ a_i^TG^kx\leq b_i  \}
\end{equation}
where $a_i^T$ and $b_i$ denote the $i$th row of $A$ and $b$, respectively.
When written in this form, it is clear that R-LD-LP is a special case of a linear semi-inifite program (see, e.g., \cite[Sect. 4]{LOPEZ2007491}). In this section, we exploit the special dynamics-based structure of R-LD-LP to characterize its complexity and determine subclasses of it that can be solved in polynomial time.}

We start with a theorem on the basic properties of \aaa{the set $\mathcal{S}$ in \eqref{eq:feasible.set.R-LD-LP}}. We recall that a set $T\subseteq\mathbb{R}^n$ is said to be \emph{invariant} under a dynamical system, if all trajectories of the dynamical system that start in $T$ remain in $T$ forever. We will simply say that a set is invariant if the underlying dynamical system is clear from the context.

\begin{theorem}\label{thm:S.closed.convex.invariant.non.polyhedral.nphard}
The set $\mathcal{S}$ in (\ref{eq:feasible.set.R-LD-LP}) has the following properties:
\begin{enumerate}[(i)]
\item It is closed, convex, and invariant.
\item It is not always polyhedral (even when $A,b,G$ have rational entries).
\item Testing membership of a given point to $\mathcal{S}$ is NP-hard.
\end{enumerate}
\end{theorem}

\begin{proof} We prove the three statements separately.

(i) Convexity and closedness are a consequence of the fact that polyhedra are convex and closed, and that (infinite) intersections of convex (resp. closed) sets are convex (resp. closed). Invariance is a trivial implication of the definition: if $x\in\mathcal{S}$, then $Gx\in\mathcal{S}$.

(ii) Clearly the set $\mathcal{S}$ can be polyhedral (consider\aaa{, e.g.,} the case where $G=I$, i.e., the identity matrix). We give an example where it is not. Consider an R-LD-LP in  $\mathbb{R}^2$ where $P=[-1,1]^2$ and 
\begin{equation}\label{eq:example2}
G=\begin{bmatrix} ~~4/5 &~ 3/5 \\-3/5 &~4/5 \end{bmatrix}
  =\begin{bmatrix} ~~\cos(\theta) &\sin(\theta) \\-\sin(\theta) &\cos(\theta) \end{bmatrix},
\end{equation}
with $\theta=\arcsin(3/5)$. This is the smallest angle (in radians) of the right triangle with sides 3, 4, and 5. From Niven's theorem \cite[Corollary 3.12]{NivensBook}, if $\bar{\theta}\in[0,\pi/2]$ and $\sin(\bar{\theta})$ is rational, then $\bar{\theta}/\pi$ is not rational unless $\bar{\theta}\in\{0,\pi/6,\pi/2\}.$ Consequently,  the number $\theta/\pi$ is irrational, which means that $\theta$ and $\pi$ are rationally independent.
Also note that $G$ is a rotation matrix that rotates points in the $xy$-Cartesian plane clockwise by $\theta$.
In other words, using polar coordinates with the convention $x_k=(r_k,\phi_k)$, the feasible region of this R-LD-LP consists of points $x_0=(r_0,\phi_0)$ such that  $(r_0,\phi_0+k\theta)\in P$ for all integers $k\geq 0$.


Notice that the closed disk $D$ centered at the origin with radius $1$ is contained in $P$ and consequently $(r_0,\phi_0+k\theta)\in P$ for all $k$ provided that $r_0\le1$.
On the other hand, consider a point $x_0=(r_0,\phi_0)\in P$ such that $r_0>1.$ Let $C$ be the circle centered at the origin with radius $r_0$. Clearly, $x_k\in C,\forall k$.
Furthermore, note that for any fixed $r_0>1$, there exist scalars $\beta_1,\beta_2,$ with $0\leq \beta_1<\beta_2<2\pi,$ such that none of the points in $C$ on the arc between $(r_0,\beta_1)$ and $(r_0,\beta_2)$  belong to $P$. However, as $\theta$ and $2\pi$ are rationally independent, we must have $\phi_0+k\theta\in (\beta_1,\beta_2)$ for some integer $k\geq 1$ (see, e.g.,~\cite[Chapter 3, Theorem~1]{CFS82}).  Consequently, $x_0$ is feasible if and only if $r_0\leq 1$, i.e., $\mathcal{S}=D$.


If we let ${S}_r\mathrel{\mathop:}=\bigcap_{k=0}^{r} \{x\in\R^2|\ G^kx\in[-1,1]^2  \}$, Figure~\ref{fig:not.polyhedral} depicts that as $r$ increases, more and more points leave the polytope $P,$ until nothing but the unit disk $D$ is left.

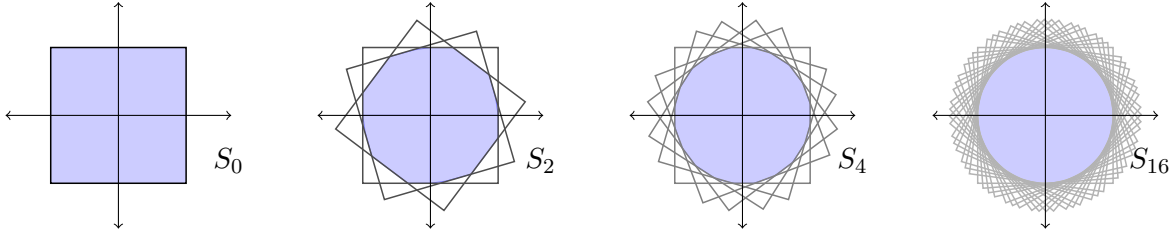
\begin{figure}[h]\centering
\begin{tikzpicture}[scale = 0.3]
\foreach \i in {0,...,0}\draw[black,fill=blue!20, line width=0.2mm, rotate around={36.87*\i:(0,0)}] (-3,-3) rectangle (3,3);
\draw [<->](-5,0) -- (5,0); \draw [<->](0,-5) -- (0,5); 
\coordinate  [label=left:$S_0$](foo) at (6,-2); 
\end{tikzpicture}
~~~~~
\begin{tikzpicture}[scale = 0.3]
\foreach \i in {0,...,3}\draw[blue!50,fill=blue!20, rotate around={90*\i:(0,0)}] (0,-3)--(-1,-3)--(-2.6,-1.8)--(-2.9,-0.7)--(-3,0)--(0,0)--cycle;
\foreach \i in {0,...,2}\draw[black!70, line width=0.2mm, rotate around={-36.87*\i:(0,0)}] (-3,-3) rectangle (3,3);
\draw [<->](-5,0) -- (5,0); \draw [<->](0,-5) -- (0,5); 
\coordinate  [label=left:$S_2$](foo) at (6,-2); 
\end{tikzpicture}
~~~~~
\begin{tikzpicture}[scale = 0.3]
\foreach \i in {0,...,3}\draw[blue!50,fill=blue!20, rotate around={90*\i:(0,0)}] (0,-3)--(-.4,-3)--(-1.2,-2.8)--(-2.1,-2.2)--(-2.7,-1.4)--(-3,-.5)--(-3,0)--(0,0)--cycle;
\foreach \i in {0,...,4}\draw[black!50, line width=0.2mm, rotate around={36.87*\i:(0,0)}] (-3,-3) rectangle (3,3);
\draw [<->](-5,0) -- (5,0); \draw [<->](0,-5) -- (0,5); 
\coordinate  [label=left:$S_4$](foo) at (6,-2); 
\end{tikzpicture}
~~~~
\begin{tikzpicture}[scale = 0.3]
\draw[yellow!20,fill=blue!20] (0,0) circle (3cm); 
\foreach \i in {0,...,16}\draw[black!30, line width=0.2mm, rotate around={36.87*\i:(0,0)}] (-3,-3) rectangle (3,3);
\draw [<->](-5,0) -- (5,0); \draw [<->](0,-5) -- (0,5); 
\coordinate  [label=left:$S_{16}$](foo) at (6,-2); 
\end{tikzpicture}
\caption{The construction in the proof of Theorem~\ref{thm:S.closed.convex.invariant.non.polyhedral.nphard}, part (ii).} \label{fig:not.polyhedral}
\end{figure}
(iii) We now prove that the following decision problem is NP-hard even when $m=1$: Given $z\in\mathbb{Q}^n, A\in\mathbb{Q}^{m\times n}, b\in\mathbb{Q}^m, G\in\mathbb{Q}^{n\times n} $, test if $z\notin\mathcal{S}$. We show this via a polynomial-time reduction from the following decision problem, which is known to be NP-hard (see~\cite[Corollary 1.2]{blondel2002presence}): Given a directed graph $\Gamma$ on $n$ nodes, test if there exists an integer $\hat{k}\geq 1$ such that $\Gamma$ has no directed paths\footnote{Following the convention of~\cite{blondel2002presence}, we allow paths to revisit the same node several times.} of length $\hat{k}$ starting from node $1$ and ending at node $n$.





Let $G\in\mathbb{Q}^{n\times n}$ be the adjacency matrix of the graph $\Gamma$, i.e., a matrix with its $(i,j)$-th entry equal to $1$ if there is an edge from node $i$ to node $j$ in $\Gamma$ and equal to $0$ otherwise. Let $z=(0,\ldots,0,1)^T\in\mathbb{Q}^n$, $A=(-1,0,\ldots,0)G\in\mathbb{Q}^{1\times n}$, and $b=-\frac{1}{2}$. Let $\mathcal{S}$ be as in (\ref{eq:feasible.set.R-LD-LP}). We claim that $z\notin\mathcal{S}$ if and only if
for some integer $\hat{k}\geq 1$, $\Gamma$ has no directed paths of length $\hat{k}$ from node $1$ to node $n$. To argue this, recall that the $(i,j)$-th entry of $G^k$ is greater than or equal to one if and only if there exists a directed path of length $k$ in $\Gamma$ from node $i$ to node $j$. Suppose first that for all integers $k\geq 1$, $\Gamma$ has a directed path of length $k$ from node 1 to node $n.$ This implies that $(1,\ldots,0)G^{k}(0,\ldots,1)^T\geq 1,$ for $k=1,2,3,\ldots$, which in turn implies that $AG^{k-1}z \leq -1$ for $k=1,2,\ldots$. As $b=-\frac{1}{2}$, we have $AG^{k}z \leq b$, for $k=0,1,\ldots,$ and hence $z \in \mathcal{S}.$ 

Suppose next that for some integer $\hat{k}\geq 1$, $\Gamma$ has no directed paths of length $\hat{k}$ from node $1$ to node $n$. This implies that $(1,\ldots,0)G^{\hat{k}}(0,\ldots,1)^T=0$ and hence we have $AG^{\hat{k}-1}z =0>b$. As $z \notin S^{\hat{k}-1}$, $z \notin \mathcal{S}$.
\end{proof}

%
%
%
%
%
%

Part (iii) of Theorem~\ref{thm:S.closed.convex.invariant.non.polyhedral.nphard} implies that in full generality, the feasible set of an R-LD-LP is unlikely to have a tractable description. The situation for particular instances of the problem may be nicer however. Let us work through a concrete example to build some geometric intuition.


\begin{Example}\label{ex:RLDLP_example}
	Consider an R-LD-LP defined by the following data:
	\begin{equation}\label{eq:example}
	A=\begin{bmatrix} -1 & 0 \\0 &-1 \\ 0 & 1\\ 1& 1\end{bmatrix}, 
	b=\begin{bmatrix} 1 \\ 1 \\ 1 \\3 \end{bmatrix}, 
	c=\begin{bmatrix} -1 \\ 0 \end{bmatrix}, 
	G=\begin{bmatrix} 0.6 & -0.4 \\0.8 &0.5 \end{bmatrix}.
	\end{equation}
	
For $k=0,1,\ldots,$	let  $P^k\mathrel{\mathop:}=\{x\in\R^2|~ AG^kx\le b\}$, so that the feasible solutions to the problem belong to the set 
	$\mathcal{S}=\bigcap_{k=0}^{\infty}P^k$.
	In Figure~\ref{fig:RLDLP.example1}, we show $P^0$,   $P^1,$ and $P^0\cap P^1.$ 
	
	\newcommand{\theaxis}
	{\draw [<->] (-2.5,0) -- (4.2,0)  node[right]{$x_1$}; \draw [<->] (0,-9.5) -- (0,4.5)  node[right]{$x_2$};}
	\newcommand{\theaxisa}
	{\draw [<->] (-2.5,0) -- (4.2,0)  node[right]{$x_1$}; \draw [<->] (0,-6.5) -- (0,3.5)  node[right]{$x_2$};}
	
	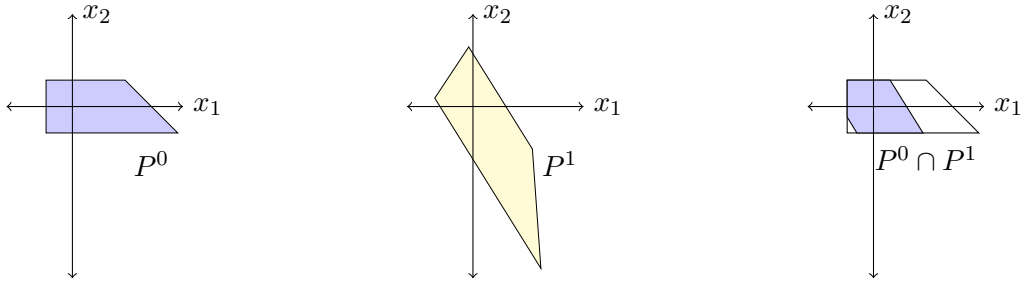
\begin{figure}[h]
		\begin{multicols}{3}
			{  \begin{tikzpicture}[scale=.35] \coordinate  [label=above:$P^1$](foo) at (.5,-1); 
				\draw [line width=.01cm, fill=blue!20] (-1,1) -- (2,1) -- (4,-1) -- (-1,-1) -- cycle ;
				\theaxisa \coordinate  [label=above:$P^0$](foo) at (3,-3); 
				\end{tikzpicture} }
			
			\columnbreak
			{  \begin{tikzpicture}[scale=.35] 
				\draw [line width=.01cm, fill=yellow!20] (-0.161,2.258) -- (2.258,-1.61) -- (2.58,-6.13) -- (-1.45,0.32) -- cycle ;
				\theaxisa \coordinate  [label=above:$P^1$](foo) at (3.3,-3); 
				\end{tikzpicture} }
			
			\columnbreak
			{  \begin{tikzpicture}[scale=.35] 
				\draw [fill=white!20] (-1,1) -- (2,1) -- (4,-1) -- (-1,-1) -- cycle ;
				\draw [fill=blue!20] (-1,1) -- (.63,1) -- (1.88,-1) -- (-.63,-1) -- ( -1,-.4) -- cycle ;
				\theaxisa \coordinate  [label=above:$P^0\cap P^1$](foo) at (2,-2.8); 
				\end{tikzpicture} }
		\end{multicols}
		\caption{ $P^0\cap P^1.$} \label{fig:RLDLP.example1}
	\end{figure}

	In Figure~\ref{fig:RLDLP.example2}, we show $P^0\cap P^1$, $P^2$, and $P^0\cap P^1\cap P^2,$ which happens to be equal to $\mathcal S.$ This is because $(P^0\cap P^1\cap P^2)\subseteq P^3$, which implies (see Lemma~\ref{lemma:termination} in Section~\ref{subsec:outer.approx.LDLP}) that $(P^0\cap P^1\cap P^2)\subseteq P^k$ for all $k>2$. 
	The optimal value for this example is achieved at the rightmost vertex  of $\mathcal{S}$ and is equal to $1.1492$. 	
	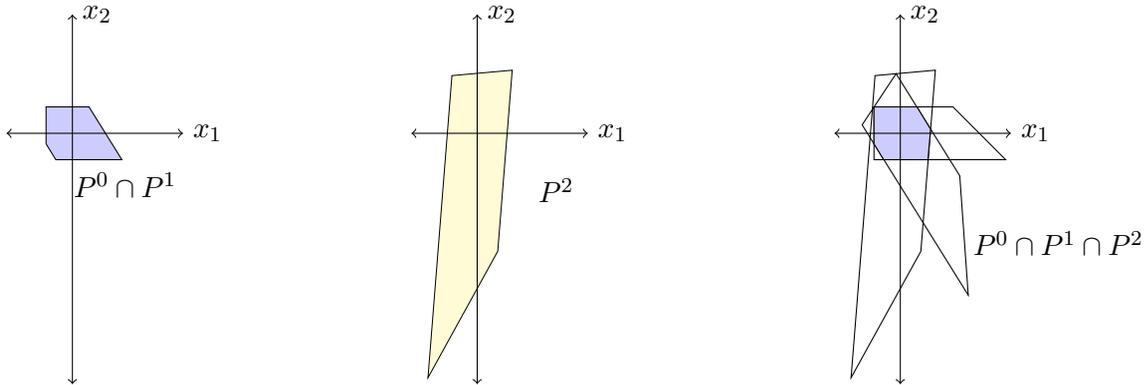
\begin{figure}[h]\begin{multicols}{3}
			{  \begin{tikzpicture}[scale=.35] 
				\draw [fill=blue!20] (-1,1) -- (.63,1) -- (1.88,-1) -- (-.63,-1) -- ( -1,-.4) -- cycle ;
				\theaxis \coordinate  [label=above:$P^0\cap P^1$](foo) at (2,-2.8); 
				\end{tikzpicture} } 
			
			\hspace*{-.2cm}
			\columnbreak
			{  \begin{tikzpicture}[scale=.35] 
				\draw [line width=.01cm, fill=yellow!20] (1.33,2.39) -- (0.78,-4.47) -- (-1.87,-9.26) -- (-0.96,2.185) -- cycle ;
				\theaxis  \coordinate  [label=above:$P^2$](foo) at (3,-3); 
				\end{tikzpicture} }

			\hspace*{.1cm}
			\columnbreak 
			{  \begin{tikzpicture}[scale=.35] 
				\draw  (-1,1) -- (2,1) -- (4,-1) -- (-1,-1) -- cycle ;
				\draw   (1.33,2.39) -- (0.78,-4.47) -- (-1.87,-9.26) -- (-0.96,2.185) -- cycle ;
				\draw  (-0.161,2.258) -- (2.258,-1.61) -- (2.58,-6.13) -- (-1.45,0.32) -- cycle ;
				\draw [fill=blue!20] (-1,1) -- (.63,1) -- (1.15,0.155) -- (1.05,-1) -- (-.63,-1) -- ( -1,-.4) -- cycle ;
				\coordinate  [label=above:$P^0\cap P^1\cap P^2$](foo) at (6,-5); 
				\theaxis
				\end{tikzpicture} }
		\end{multicols}
		\caption{ $\mathcal S=P^0\cap P^1\cap P^2.$} \label{fig:RLDLP.example2}
	\end{figure}		
\end{Example}

\subsection{Outer approximations to $\mathcal{S}$}\label{subsec:outer.approx.LDLP}
For an integer $r\geq 0,$ let
\begin{equation}\label{eq:Sr.LDLP}
{S}_r\mathrel{\mathop:}=\bigcap_{k=0}^{r} \{x\in\R^n|\ AG^kx\leq b  \}.
\end{equation}
In view of the definition of $\mathcal{S}$ in (\ref{eq:feasible.set.R-LD-LP}), we have $$\mathcal{S}\subseteq \ldots S_{r+1}\subseteq S_r\subseteq \ldots\subseteq S_2\subseteq S_1\subseteq S_0=P.$$
In other words, the sets $S_r$ provide polyhedral outer approximations to the feasible set \aaa{$\mathcal{S}$} of an R-LD-LP, which get tighter as $r$ increases. \aaa{It can be shown that the sets $S_r$ converge\footnote{\aaa{By convergence of the sequence $\{ S_r \}$ to $\mathcal{S}$, we mean (see \cite{Royset2020}) that $\forall x\in\mathbb{R}^n, \lim_{r\rightarrow\infty}\text{dist}(x,S_r)= \text{dist}(x,\mathcal{S}).$ Here, for a point  $x\in\mathbb{R}^n$ and a set $\Omega\subseteq \mathbb{R}^n$, $\text{dist}(x,\Omega)$ is defined as $\inf_{y\in\Omega} ||y-x||$.}} to $\mathcal{S}$ in the limit. By solving linear programs that minimize $c^Tx$ over $S_r$, one obtains a nondecreasing sequence of lower bounds on the optimal value of our R-LD-LP. Moreover, one can show that if $S_r$ is bounded for some $r$, this sequence converges to the optimal value of the R-LD-LP as $r\rightarrow\infty$.\footnote{\aaa{The assumption on boundedness of $S_r$ for some $r$ cannot be dropped as one can verify by considering the R-LD-LP instance given by $$A=[1~-1],\quad
b=1,\quad
c=[-1~0]^T,\quad 
G~=~\begin{bmatrix}1 & ~0 \\0 &~1/2\end{bmatrix}.$$ Note also that when $P$ is bounded, this assumption is clearly satisfied.}}}




 
We have already seen an example where convergence of $S_r$ to $\mathcal{S}$ is finite (Example~\ref{ex:RLDLP_example}), and one where it is not (cf. Theorem~\ref{thm:S.closed.convex.invariant.non.polyhedral.nphard}, part (ii)). Our goal in this subsection is to study the following two questions:

\begin{enumerate}[(i)]
\item What are the barriers that can prevent convergence of $S_r$ to $\mathcal{S}$ from being finite \aaa{(i.e., make $S_r\setminus\mathcal{S}\neq \emptyset$ for all positive integers $r$)?}
\item When convergence of $S_r$ to $\mathcal{S}$ is finite, \aaa{i.e.,} when $\mathcal{S}=S_{r^*}$ for some positive integer $r^*$, can we provide an efficiently computable upper bound on $r^*$?
\end{enumerate}

In regards to question (i), we show that there are three separate barriers to finite convergence \aaa{(meaning that in the presence of any of these conditions convergence may or may not be finite):} the matrix $G$ having spectral radius larger or equal to 1 (Proposition~\ref{prop:rho>1.convergence.not.finite} and Theorem~\ref{thm:S.closed.convex.invariant.non.polyhedral.nphard}, part (ii)), the origin being on the boundary of the polyhedron $P$ (Proposition~\ref{prop:convergence.not.finite.without.origin}), and the polyhedron $P$ being unbounded (Proposition~\ref{prop:P.unbounded.not.finite}). In regards to question (ii), we show that once these three barriers are removed, then $S_r$ reaches $\mathcal{S}$ in a number of steps that is not only finite, but upper bounded by a quantity that can be computed in polynomial time (Theorem~\ref{thm.polyhedrality}).

Before we prove these results, let us start with a simple lemma that allows us to detect finite termination.

\begin{lemma} \label{lemma:termination}
If ${S}_r=  {S}_{r+1}$  for some integer $r\ge0$, then $\mathcal S = {S}_{r}$.
Furthermore, for any fixed $r\geq 0$, the condition ${S}_r=  {S}_{r+1}$ can be checked in polynomial time.
\end{lemma}

\begin{proof}
We first observe that condition ${S}_r=  {S}_{r+1}$ implies that the set ${S}_r$ is invariant. 
If not, there would exist an $x\in {S}_r$ with $Gx\notin {S}_r$. But this implies that $x\notin {S}_{r+1}$, which is a contradiction. 
Invariance of $S_r$ implies that ${S}_r=\mathcal{S}$ and the first part of the claim is proven.

To show the second part of the claim, note that ${S}_{r+1}\subseteq {S}_{r}$ by definition, so only the reverse inclusion needs to be checked. For $i\in\{1,\ldots,m\},$ let $a_i$ denote the transpose of the $i$-th row of $A$. We can then solve, in polynomial time, $m$ linear programs   
\begin{equation}
\max_{x\in\R^n}\{a_i^TG^{r+1}x: x\in S_r\},
\end{equation}
and declare that ${S}_{r}= {S}_{r+1}$ if and only if the optimal value of the $i$-th program is less than or equal to $b_i$ for all $i\in\{1,\ldots,m\}$. 
\end{proof}

In view of this lemma, the reader can see that in Example~\ref{ex:RLDLP_example}, the observation that $S_2$ equals $S_3$ allowed us to conclude that $\mathcal{S}$ equals $S_2.$ We now characterize scenarios where convergence of $S_r$ to $\mathcal{S}$ is not finite. Note that this is equivalent to having $S_{r+1}\subset S_{r}$ for all $r\geq 0.$

 
Recall that the \emph{spectral radius} $\rho (G)$ of an $n\times n$ matrix $G$ is given by
$$\rho (G)=\max\{|\lambda_1|, \ldots, |\lambda_n|\},$$
where $\lambda_1, \ldots, \lambda_n$ are the (real or complex) eigenvalues of $G$. Theorem~\ref{thm:S.closed.convex.invariant.non.polyhedral.nphard}, part (ii) has already shown that when $\rho(G)=1,$ convergence of $S_r$ to $\mathcal{S}$ may not be finite. The following simple construction shows that the same phenomenon can occur when $\rho(G)>1,$ even when the set $\mathcal{S}$ is polyhedral.

\newcommand{\theaxis}{\draw [<->] (-2.5,0) -- (12.5,0)  node[right]{$x_1$}; \draw [<->] (0,-2.5) -- (0,12.5)  node[left]{$x_2$};
 \coordinate  [label=left:1](foo) at (0,10);  
 \coordinate  [label=below:-1](foo) at (1,-.3);\coordinate  [label=below:0](foo) at (5.5,-.3);\coordinate  [label=below:1](foo) at (10.5,-.3); 
 \draw (10,-.3)--(10,.3); \draw (5,-.3)--(5,.3); \draw (-.3,10)--(.3,10);}

\begin{figure}[h]\centering
\begin{tikzpicture}[scale = 0.3]
\draw[black, line width=.01cm, fill=orange!70] (-3,-3) rectangle (3,3);
\draw [<-> ](-5,0) -- (5,0); \draw [<->](0,-5) -- (0,5); 
\coordinate  [label=left:$S_0$](foo) at (6,2); 
\end{tikzpicture}
~~~~~
\begin{tikzpicture}[scale = 0.3]
\draw[black, line width=.01cm, fill=orange!70] (-1.5,-3) rectangle (1.5,3);
\draw [<->](-5,0) -- (5,0); \draw [<->](0,-5) -- (0,5); 
\coordinate  [label=left:$S_1$](foo) at (5,2); 
\end{tikzpicture}
~~~~~
\begin{tikzpicture}[scale = 0.3]
\draw[black!30, line width=.01cm, fill=orange!70] (-.21,-3) rectangle (.21,3);
\draw [<->](-5,0) -- (5,0); \draw [<->](0,-5) -- (0,5); 
\coordinate  [label=left:$S_4$](foo) at (4,2); 
\end{tikzpicture}
~~~~~
\begin{tikzpicture}[scale = 0.3]
\draw [<->](-5,0) --(5,0); 
\draw [<-](0,-5) -- (0,-3); \draw [->](0,3) -- (0,5); 
\draw[line width=.05cm, orange!70] (0,-3) --(0,3);
\coordinate  [label=left:$\mathcal S$](foo) at (3,2); 
\end{tikzpicture}
~~~~~
\caption{ The construction in the proof of Proposition~\ref{prop:rho>1.convergence.not.finite}.
} \label{fig:convergnce.not.finite}
 \end{figure}
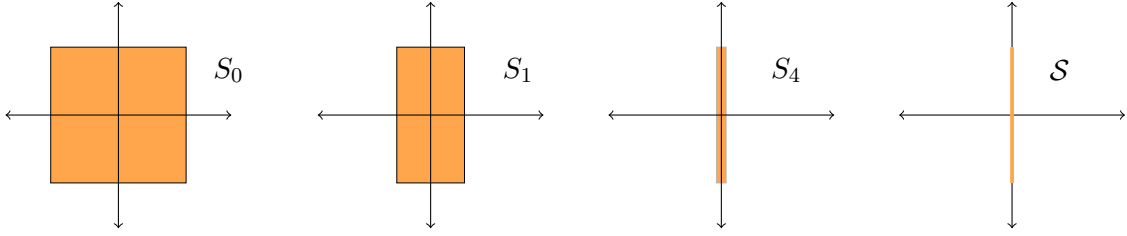

\begin{proposition}\label{prop:rho>1.convergence.not.finite}
If $\rho(G)>1,$ then convergence of ${S}_r$ to $\mathcal{S}$  may not be finite even when $P$ is a bounded polyhedron that contains the origin in its interior. 
\end{proposition}

\begin{proof}
Let $a>0$ and consider an instance of an R-LD-LP with $P=[-1,1]^2$ and $G~=~\begin{pmatrix}a & ~0 \\0 &~1/a\end{pmatrix}$. 
It is easy to see that for $r\geq 0$, $$S_r=\{x\in\R^2 | \   -a^{-r}\le x_1\le a^{-r}, -1\le x_2\le 1\}.$$
Hence, the set $\mathcal{S}$ is the line segment joining the points $(0,-1)$ and $(0,1)$, and convergence of $S_r$ to $\mathcal{S}$ is not finite. See Figure~\ref{fig:convergnce.not.finite} for an illustration.
\end{proof}

 \begin{proposition}\label{prop:convergence.not.finite.without.origin}
If the origin is not contained in the interior of $P,$ then convergence of ${S}_r$ to $\mathcal{S}$  
may not be finite even when $\rho(G)<1$ and $P$ is bounded. 
\end{proposition}
	\begin{proof}
		Consider an instance of an R-LD-LP in $\R^2$ with $P=[0,1]^2$ and 
		\begin{equation}\label{eq:G.not.finite.exmaples}
		G=\frac{1}{2}\begin{bmatrix} ~~2/3 &~ -1/3 \\-1/3 &~~~~2/3 \end{bmatrix}.
		\end{equation}
		Note that $\rho(G)=1/2$ and the origin is contained in $P$, but not in the interior of $P$.
		It can  be checked that 
		\begin{equation}\label{eq:G^k.not.finite.exmaples}G^k
		=\frac{1}{2^{k+1}}\left(\begin{bmatrix} ~~1 &~ -1 \\[.2cm]-1 &~~~~~1 \end{bmatrix}
		+\frac1{3^k}\begin{bmatrix} {~1~~} &~~1~ \\[.2cm]~1~~ &~~1~ \end{bmatrix}\right)
		\end{equation}
		for any integer $k\geq 1.$ It follows that 
		\begin{align} \label{eq:impl.lemma.1}
	G^kx\ge0\Longleftrightarrow |x_1-x_2|\le (1/3^k)(x_1+x_2),
		\end{align} 
		and therefore $\{x\in\R^2|~ G^{k+1}x\ge0\} \subset \{x\in\R^2|~ G^kx\ge0\}$ for any  integer $k\ge1$. 
		Similarly,
				\begin{align} \label{eq:impl.lemma.1b}
		1\ge G^kx\Longleftrightarrow 2^{k+1} \ge |x_1-x_2|+ (1/3^k)(x_1+x_2), \forall k\geq 1.
		\end{align} 
Observe that for any $k\ge1,$ if $x\in P=[0,1]^2$, then $2^{k+1} \ge |x_1-x_2|+ (1/3^k)(x_1+x_2)$.
		Combining this observation with \eqref{eq:impl.lemma.1} and \eqref{eq:impl.lemma.1b}, we get that
		\begin{equation}
		{S}_r=\bigcap_{k=0}^{r} \{x\in\R^n|\ G^kx\in P  \} = P\cap \{x\in\R^n|\ |x_1-x_2|\le (1/3^r)(x_1+x_2) \},
		\end{equation}
		and therefore ${S}_r$ strictly contains ${S}_{r+1}$ for all  $r\ge1$.
		This shows that convergence of $S_r$ to $\mathcal{S}=\{x\in\R^2|\ 0\leq x_1\leq 1,x_2=x_1\}$ cannot be finite. Figure \ref{fig:originneeded} demonstrates this asymptotic convergence.
		\end{proof}

\renewcommand{\theaxis}{\draw [<->] (-2.5,0) -- (12.5,0)  node[right]{$x_1$}; \draw [<->] (0,-2.5) -- (0,12.5)  node[left]{$x_2$};
 \coordinate  [label=left:1](foo) at (0,10);  \coordinate  [label=below:1](foo) at (10,-.3); 
 \draw (10,-.3)--(10,.3); \draw (-.3,10)--(.3,10);}

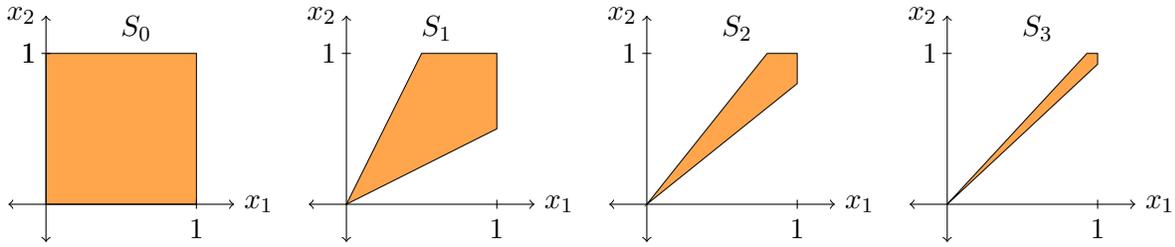
\begin{figure}[h]
\begin{multicols}{4}
{  \begin{tikzpicture}[scale=.2] 
 \draw [line width=.01cm, fill=orange!70] (0,0) -- (0,10) -- (10,10) -- (10,0) -- cycle ;
   \coordinate  [label=above:$S_0$](foo) at (6,10.2); \theaxis    \coordinate  [label=left:1](foo) at (0,10); 
 \end{tikzpicture} }
 
\columnbreak
{  \begin{tikzpicture}[scale=.2] 
 \draw [line width=.01cm, fill=orange!70] (0,0) -- (5,10) -- (10,10) -- (10,5) -- cycle ;
   \coordinate  [label=above:$S_1$](foo) at (6,10.2); \theaxis
 \end{tikzpicture} }
 
\columnbreak 
{  \begin{tikzpicture}[scale=.2] 
 \draw [line width=.01cm, fill=orange!70] (0,0) -- (40/5,10) -- (10,10) -- (10,40/5) -- cycle ;
   \coordinate  [label=above:$S_2$](foo) at (6,10.2); \theaxis
 \end{tikzpicture} }
 
\columnbreak 
{  \begin{tikzpicture}[scale=.2] 
 \draw [line width=.01cm, fill=orange!70] (0,0) -- (130/14,10) -- (10,10) -- (10,130/14) -- cycle ;
  \coordinate  [label=above:$S_3$](foo) at (6,10.2); \theaxis
 \end{tikzpicture} }
 
  \end{multicols}
  \caption{The construction in the proof of Proposition~\ref{prop:convergence.not.finite.without.origin}.} \label{fig:originneeded}
  \end{figure}

\begin{proposition}\label{prop:P.unbounded.not.finite}
	If $P$ is unbounded, then convergence of $S_r$ to $\mathcal{S}$ may not be finite even if $\rho(G)<1$ and the origin is in the interior of $P$.
\end{proposition}

\begin{proof}
	Consider an instance of R-LD-LP in $\mathbb{R}^2$ with $P=\{(x_1,x_2) \in \mathbb{R}^2|~ x_1\geq -1\}$ and $G$ as in (\ref{eq:G.not.finite.exmaples}). Note that $P$ is a half-space that contains the origin in its interior and $\rho(G)=1/2.$ Recall our notation $P^k=\{x \in \mathbb{R}^n|~G^kx \in P\}$ and $S_r=\cap_{k=0}^r P^k$. Our goal is to provide a sequence of points $\{z_r\}_{r\geq 1}$ such that $z_r \in P^k, \forall k=0,\ldots,r,$ but $z_r \notin P^{r+1}$. This would imply that $S_{r+1}$ is a strict subset of $S_{r}$ for all $r\geq 1$, which shows that convergence of $S_r$ to $\mathcal{S}$ cannot be finite.
	
We first start by characterizing the sets $P^k$ for $k\geq 1.$ In view of (\ref{eq:G^k.not.finite.exmaples}), we have
	$$G^k\begin{pmatrix} x_1 \\ x_2 \end{pmatrix}=\frac{1}{2^{k+1}} \begin{pmatrix} x_1(1+\frac{1}{3^k})+x_2(\frac{1}{3^k}-1)\\ x_1(-1+\frac{1}{3^k})+x_2(\frac{1}{3^k}+1)\end{pmatrix}.$$
	Hence, it is easy to check that for $k\geq 1$, $P^k$ is the following half-space:
	\begin{align*}
	P^k=\{(x_1,x_2)\in \mathbb{R}^2|~x_2\leq l_k x_1+i_k\}, \text{ where } l_k\mathrel{\mathop{:}}=\frac{3^k+1}{3^k-1} \text{ and } i_k\mathrel{\mathop{:}}=\frac{2\cdot 6^k}{3^k-1}.
	\end{align*}
We now define, for $r\geq 1$, the coordinates of the points $z_r=(z_{r,1},z_{r,2})^T$ to be
	\begin{align*}
	z_{r,1}=\frac{i_{r+2}-i_r}{l_{r}-l_{r+2}} \text{ and } z_{r,2}=i_{r+2}+l_{r+2} \cdot \left( \frac{i_{r+2}-i_r}{l_{r}-l_{r+2}} \right).
	\end{align*}
	 Fix $r \geq 1$. Note that $z_{r,1}=2^{r-3}(3^{r+3}-35) \geq \frac14 (3^4-35) \geq -1,$ which implies that $z_r \in P.$ To show that $z_r \in P^k, \forall k=1,\ldots,r$, we need to show that $z_{r,2} \leq l_kz_{r,1} +i_k,$ for $k=1,\ldots,r.$ This is the same as showing that $$\frac{i_{r+2}-i_r}{l_r-l_{r+2}} \geq \frac{i_{r+2}-i_k}{l_k-l_{r+2}}, \mbox{ for } k =1,\ldots,r,$$ which is equivalent to showing that the difference of the two ratios, i.e.,
	\begin{align}\label{eq:proof.algebra}
	\frac{(3^{r+2}-1)(2^{k+3}3^k+2^r3^{r+3}-2^{r+3}3^k-2^r3^{k+3})}{2^3(3^{r+2}-3^k)}
 \end{align}
	is nonnegative for $k=1,\ldots,r$. As we have $$2^{k+3}3^k+2^r3^{r+3}-2^{r+3}3^k-2^r3^{k+3}=2^r3^k(2^{3-(r-k)}+3^{3+(r-k)}-2^3-3^3)$$
	and as $x \geq 0 \Rightarrow 2^{3-x}+3^{3+x}-2^3-3^3 \geq 0$, we have $z_r \in P^{k}, \forall k=1,\ldots,r.$
	
	To show that $z_{r} \notin P^{r+1}$, we simply need to show that 
	$$\frac{i_{r+2}-i_r}{l_r-l_{r+2}} < \frac{i_{r+2}-i_{r+1}}{l_{r+1}-l_{r+2}}.$$
	Replacing $k$ by $r+1$ in (\ref{eq:proof.algebra}), this is equivalent to showing that
	$$2^{r+4}3^{r+1}+2^r3^{r+3}-2^{r+3}3^{r+1}-2^r3^{r+4} < 0.$$ 
	As $$2^{r+4}3^{r+1}+2^r3^{r+3}-2^{r+3}3^{r+1}-2^r3^{r+4}=2^r3^{r+1}(2^4+3^2-2^3-3^3)=-10\cdot2^r3^{r+1},$$ this inequality clearly holds.
\end{proof}

\renewcommand{\theaxis}{
	\draw [<->, color=black] (-5,5) -- (50,5)  node[right]{$x_1$}; \draw [<->] (7,-2) -- (7,59)  node[right]{$x_2$};
}

\begin{figure}[h]
	\begin{multicols}{4}
		{  \begin{tikzpicture}[scale=.04] 
			\shade[left color=orange!70,right color=orange!30] (0,-3)--(0,62)--(62,62)--(62,-3)--cycle; 
			\draw (0,-4)--(0,62) ;\coordinate[label=left:$S_0$](foo) at (26,45);	\theaxis   \end{tikzpicture} }
		
		{  \begin{tikzpicture}[scale=.04] 
			\shade[left color=orange!70,right color=orange!30] (0,-3)--(0,6)--(33,62) --(62,62)--(62,-3)--cycle; 
			\draw (0,-4)--(0,62);\draw (0,6)--(33,62); \coordinate[label=right:$S_1$](foo) at (33,52);	\theaxis   \end{tikzpicture} }
		
		{  \begin{tikzpicture}[scale=.04] 
			\shade[left color=orange!70,right color=orange!30] (0,-3)--(0,6)--(13,28)--(41,62) --(62,62)--(62,-3)--cycle; 
			\draw (0,-4)--(0,62) ;\draw(0,6)--(33,62) ;\draw (0,13)--(41,62); \coordinate[label=right:$S_2$](foo) at (41,52);
			\theaxis    \end{tikzpicture} }
		
		{  \begin{tikzpicture}[scale=.04] 
			\shade[left color=orange!70,right color=orange!30] (0,-3)--(0,6)-- (13,28)--(22,40)--(48,67) --(53,67)--(62,67)--(62,-3)--cycle; 
			\draw (0,-4)--(0,62) ;\draw(0,6)--(33,62) ;\draw(0,13)--(41,62) ;\draw(0,18)--(48,67); \coordinate[label=right:$S_3$](foo) at (48,52);
			\theaxis 	\end{tikzpicture} }
		
	\end{multicols}
	\caption{The construction in the proof of Proposition~\ref{prop:P.unbounded.not.finite}.} \label{fig:unbounded.bad}
\end{figure}
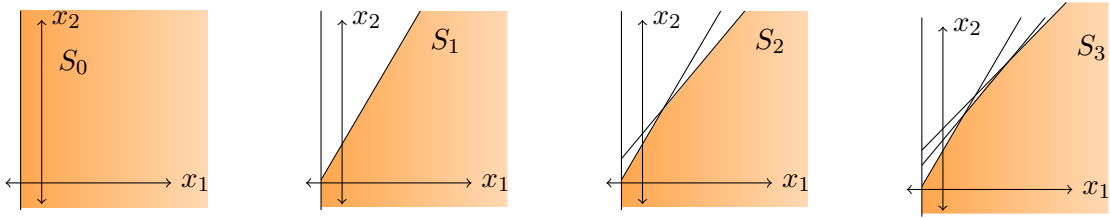

\subsubsection{A polynomially-computable upper bound on the number of steps to convergence}\label{subsubsec:finite.LDLP.}
\aaa{In this subsection, we show that when (i) $\rho(G)<1$, (ii) $P$ is bounded, and (iii) the origin is in the interior of $P$, then $\mathcal{S}=S_r$ for some integer $r$ that can be computed in time polynomial in the size of the R-LD-LP data. Note that in view of Propositions~\ref{prop:rho>1.convergence.not.finite},~\ref{prop:convergence.not.finite.without.origin},~\ref{prop:P.unbounded.not.finite}, when one of the assumptions (i), (ii), (iii) above does not hold, one can find examples where convergence of $S_r$ to $S$ is not even finite.}

\st{Propositions~2.4, 2.5,2.6 show that there are three necessary conditions for guaranteed finite convergence of the sets $S_r$ in (eq:Sr.LDLP) to the set $\mathcal{S}$ in (eq:feasible.set.R-LD-LP): having (i) $\rho(G)<1$, (ii) $P$ bounded, and (iii) the origin in the interior of $P$. In this subsection, we show that these three conditions are also sufficient for finite convergence, and that an upper bound on the number of steps to convergence can be computed in time polynomial in the size of the data. }


Arguably, the condition $\rho(G)<1$ accounts for one of the more interesting settings of an R-LD-LP. Indeed, if $\rho(G)>1,$ then trajectories of the dynamical system $x_{k+1}=Gx_k$ starting from all but a measure zero set of initial conditions go to infinity and hence, at least when the polyhedron $P$ is bounded, the feasible set $\mathcal S$ of our R-LD-LP can never be full dimensional. The boundary case $\rho(G)=1$ is more delicate. Here, the system trajectories can stay bounded or go to infinity depending on the geometric/algebraic multiplicity of the eigenvalues of $G$ with absolute value equal to one. Even in the bounded case, we have shown already in Theorem~\ref{thm:S.closed.convex.invariant.non.polyhedral.nphard}, part (ii) that the feasible set of an R-LD-LP may not be polyhedral. Hence the optimal value of an R-LD-LP may not even be a rational number (consider, e.g., the set $\mathcal{S}$ associated with Figure~\ref{fig:not.polyhedral} with $c=(1,1)^T$). Note also that when $\rho(G)<1$, we have $\lim_{k\rightarrow\infty} G^kx_0=0$ for all $x_0\in\mathbb{R}^n$ and hence if the origin is not in $P$, then the feasible set of the R-LD-LP is empty. As a consequence, the assumption that the origin be in $P$ is reasonable and that it be in the interior of $P$ is only slightly stronger (and cannot be avoided).


For the convenience of the reader, we next give the  standard definitions (see, e.g., \cite{Schrijver_LP_IP_Book}) on sizes of rational data  that we will use in the following result.
We say that the size of a rational number $r = p/q$ where $p,q\in \mathbb{Z}$ (and are relatively prime), is $1 + \lceil \log_2(|p| + 1)\rceil + \lceil \log_2(|q| + 1)\rceil$. 
We denote the size of $r$ by $\sigma(r)$ and note that $1/|r|,|r|\leq 2^{\sigma(r)}$.
Similarly, the size of a rational vector (or matrix) is defined to be the sum of the sizes of its components plus the product of its dimensions. 
It is well known that multiplying two matrices gives a matrix of size polynomially bounded by the sizes of the initial matrices.  
The inverse of a nonsingular rational matrix has size polynomially bounded by the size of the matrix. 
Similarly, any system of rational linear equations has a solution of size polynomially bounded by the size of the data defining the equality system \cite{Schrijver_LP_IP_Book}.
In addition, this solution can be computed in polynomial time.
Consequently, if a linear program defined by rational data has an optimal solution, then it has one with size polynomial in the data defining the LP.
Clearly the optimal value of the LP has size polynomial in the data as well. Finally, we remark that given a polyhedron $P=\{x\in\mathbb{R}^n|~ Ax\leq b  \}$, one can check whether $P$ contains the origin in its interior and whether $P$ is bounded in time polynomial in $\sigma(A,b)$. The former task simply requires checking if the entries of $b$ are all positive, and the latter can be carried out e.g. by minimizing and maximizing each coordinate $x_i$ over $P$ and checking if the optimal values of the resulting LPs are all finite. One can also check if the spectral radius of a square matrix $G$ is less than one in time polynomial in $\sigma(G)$~\cite[Section 2.6]{blondel2000survey}.



%


{\cred 
\begin{theorem}\label{thm.polyhedrality}
Let $\sigma(A,b,G)$ denote the size of $A$, $b,$ and $G$. Let $\mathcal{S}$ and $S_r$ be as in (\ref{eq:feasible.set.R-LD-LP}) and (\ref{eq:Sr.LDLP}) respectively. If $\rho(G)<1,$ \aaa{$P=\{x\in\mathbb{R}^n|~ Ax\leq b  \}$} is bounded, and the origin is in the interior of $P$, then $\mathcal S=S_r$ for some nonnegative integer $r$ that can be computed in time polynomial in $\sigma(A,b,G)$. Furthermore, for any fixed rational number ${\rho^*}<1$, R-LD-LPs with $\rho(G)\leq {\rho^*}$ can be solved in time polynomial in $\sigma(A,b,c,G,\rho^*)$.
\footnote{Our preliminary version of this work~\cite{rdo_cdc15} unfortunately did not have the assumption that $\rho^*$ be fixed, which is needed in our proof. The statement in the proof of Theorem 3.1 of~\cite{rdo_cdc15} that the integer $r$ computed there has polynomial size is correct, but that linear optimization over $S_r$ is an LP of polynomial size is incorrect. We ask that the reader refer to the statement and proof of Theorem~\ref{thm.polyhedrality} of the current paper instead.}

\end{theorem}
}
\begin{proof}
\aaa{For the first claim, we prove that $r$ can be computed in polynomial time by paying attention to the bit-size complexity of the numbers involved in each of the following four steps:}
\begin{itemize}
\item \aaa{ First, we find a positive definite matrix $M\succ 0$ 
\footnote{We use the standard notation $A\succ 0$ to denote that a matrix $A$ is positive definite (i.e., has positive  eigenvalues), and  $A\succeq 0$ to denote that $A$ is positive semidefinite (i.e., has nonnegative eigenvalues).} 
and ellipsoids $E(\alpha)= \{x\in\R^n|~x^TMx\leq \alpha  \}$ that are invariant under the dynamics $x_{k+1}=Gx_k$.}
	\item Next, we find scalars $\alpha_1,\alpha_2>0$ \aaa{such that $E(\alpha_1)\subseteq P\subseteq E(\alpha_2)$.}
	\item Then we compute a ``shrinkage factor'' $\gamma\in(0,1)$, which gives a lower bound on the amount our ellipsoids $E(\alpha)$ shrink under one application of $G$. 
	\item Finally, using $\alpha_1,\alpha_2,\gamma$, we compute a nonnegative integer $r$ such that $G^r E(\alpha_2)\subseteq E(\alpha_1).$ This will imply that $\mathcal S=S_r$.
	\end{itemize}

 \aaa{The second claim of the theorem then follows by adapting Steps 1 and 3 exploiting the fact that $\rho^*<1$ is fixed.}


\noindent{\bf Step 1. Computing an invariant ellipsoid.} 
To find an invariant ellipsoid for the dynamical system $x_{k+1}=Gx_k$, we solve the well-known Lyapunov equation
\begin{equation}\label{eq:Lyapunov.linear.system} G^TMG-M=-I, \end{equation}
for the symmetric matrix $M$, where $I$ here is the $n\times n$ identity matrix. Since $\rho(G)<1$, this linear system is guaranteed to have a unique solution (see e.g.~\cite[Chap. 14]{dahleh2004lectures}). 
Note that $M$ is rational, can be computed in polynomial time, and has size polynomially bounded by $\sigma(G).$
%
%
Further, we claim that $M$ must be positive definite. To see this, suppose we had $y^TMy\leq 0$, for some $y\in\mathbb{R}^n, y\neq 0.$ Multiplying (\ref{eq:Lyapunov.linear.system}) from left and right by $y^T$ and $y$, we see that $y^TG^TMGy\leq -y^Ty~<~0$. In fact, $y^T(G^k)^TMG^ky\leq -y^Ty<0$, for all $k\geq 1$. But since $\rho(G)<1,$ we must have $G^ky\rightarrow 0$ as $k\rightarrow \infty$, and hence $y^T(G^k)^TMG^ky\rightarrow 0$ as $k\rightarrow \infty$, a contradiction.



Since $M\succ 0$, the sets $$E(\alpha)\mathrel{\mathop:}=\{x\in\R^n|~x^TMx\leq \alpha\}$$ are bounded for all $\alpha\geq 0.$ Furthermore, because $G^TMG\prec M$ in view of (\ref{eq:Lyapunov.linear.system}), if $x\in E(\alpha)$  for some $\alpha\geq 0$, then $Gx\in E(\alpha),$ and hence $G^k x\in E(\alpha)$ for all integers $k\ge0$.


\noindent{\bf Step 2. Computing inner and outer ellipsoids.} 
We next compute scalars $\alpha_2>\alpha_1>0$ such that the ellipsoids $E(\alpha_1)=\{x\in\R^n|~x^TMx\leq \alpha_1  \}$ and $E(\alpha_2)=\{x\in\R^n|~x^TMx\leq \alpha_2  \}$ satisfy
$$E(\alpha_1) \subseteq  P \subseteq E(\alpha_2).$$

For $i=1,\ldots,m$, let $a_i^Tx\le b_i$ denote the $i$-th defining inequality of $P$. For each $i$, compute a scalar $\eta_i$ as the optimal value of the following convex program:
 $$\eta_i\mathrel{\mathop:}=\max_{x\in\R^n}\{a_i^T x: x^TMx\leq 1\}.$$ 
The optimal value here can be computed in closed form: $$\eta_i=\sqrt{a_i^TM^{-1}a_i}.$$
Note that $M^{-1}$ exists as $M\succ 0$. We then let 
\beqn \alpha_1=\min_{i\in\{1,\ldots,m\}}\left\{{b_i^2}/{\eta_i^2}\right\}=\min_{i\in\{1,\ldots,m\}}\left\{\frac{b_i^2}{a_i^TM^{-1}a_i}\right\}.\label{eq:alpha1}\eeqn 
This ensures that $\forall i\in\{1,\ldots,m\}$, $\forall x\in E(\alpha_1)$, we have $a_i^Tx\le b_i$. 
Hence, $E(\alpha_1)\subseteq P.$ Note that $\alpha_1>0$ since $M^{-1}\succ 0,$ and $b_i>0$ for $i=1,\ldots,m$ as the origin is in the interior of $P$.
%
Moreover, the size of  $\alpha_1$ is polynomially bounded by $\sigma(A,b,G)$ as the size of the numbers used in the calculation are all polynomially bounded by  $\sigma(A,b,G)$.

We next compute a scalar $\alpha_2>0$ such that $P\subseteq E(\alpha_2). $
As $P$ is a polytope,  $$P\subseteq \{x\in\R^n|~  l_i\leq x_i \leq u_i \},$$ 
where the scalars $u_i,l_i\in\R,$ for $i=1,\ldots, n$,  can be obtained by solving  LPs that minimize/ maximize each coordinate $x_i$ over $P$.
Note that this can be done in polynomial time.
We can now bound $x^TMx=\sum_{i,j} M_{i,j}x_ix_j$ term by term to  obtain
\beqn\alpha_2=\sum_{i=1}^n \sum_{j=1}^n\max\{|M_{i,j}u_iu_j|, |M_{i,j}l_il_j|, |M_{i,j}u_il_j|, |M_{i,j}l_iu_j|   \}.\label{eq:alpha2}\eeqn 
Clearly for any $x\in P$, we have $x^TMx\le\alpha_2$ and therefore $P\subseteq E(\alpha_2)$.
As optimal values of linear programs are bounded polynomially by the size of their data, we have  the size of  $\alpha_2$ polynomially  bounded by  $\sigma(A,b,G)$.

\noindent{\bf Step 3. Computing a shrinkage factor.} 
Next we argue that the number 
\beqn\gamma=1-\frac{1}{\max_{i\in\{1,\ldots,n\}}\{M_{ii}+\sum_{j\neq i}|M_{i,j}|\}}\label{eq:shrinkage}\eeqn
satisfies $G^TMG\preceq \gamma M$.
Observe that for any $x\in\R^n,$ the Lyapunov equation (\ref{eq:Lyapunov.linear.system}) implies 
$$x^TG^TMGx ~~=~~x^TMx-x^Tx ~~\leq~~(1-\eta) x^TMx ,$$
provided that $\eta>0$ is a scalar that satisfies the inequality   $$\eta x^TMx\leq x^Tx$$
for all $x\in\R^n.$ Let $\lambda_{max}(M)$ denote the largest eigenvalue of the matrix $M$.
Note that any $\eta\le{1}/{\lambda_{max}(M)}$ satisfies the above inequality.
By Gershgorin's circle theorem, we have
$$\lambda_{max}(M)\leq \max_{i\in\{1,\ldots,n\}}\{M_{ii}+\sum_{j\neq i}|M_{i,j}|\}.$$
Using this upper bound, we establish that  $G^TMG\preceq \gamma M$ for $\gamma$ as given in (\ref{eq:shrinkage}).

We observe that as $M\succ0$, and hence $G^TMG\succeq 0,$ we have $M-I=G^TMG\succeq 0$.
This implies that $M_{ii}\geq 1$ for all $i\in\{1,\ldots,n\}$ and therefore $\gamma$ is indeed a number in $[0,1)$. Note also that the size of $\gamma$ is polynomially bounded by $\sigma(G).$


{\bf Step 4. Computing the number of steps to convergence.}
Using $\alpha_1,\alpha_2$ and $\gamma$, we now compute an integer $\bar{r}$ such that $\gamma^{\bar{r}}E(\alpha_2)\subseteq E(\alpha_1)$ and therefore all  points inside the outer ellipsoid  $E(\alpha_2)$ are guaranteed to be within the inner ellipsoid $E(\alpha_1)$ after at most $\bar{r}$ steps.

As $G^TMG\preceq \gamma M$, we have $(G^k)^TMG^k\preceq \gamma^k M$.
Hence, if $x\in E(\alpha_2)$, i.e., $x^TMx \leq \alpha_2$, then $x^T(G^k)^TMG^kx\le \gamma^k\alpha_2$.
Clearly  $\gamma^{\bar{r}}\alpha_2\le\alpha_1$ for 
\beq \bar{r}=\lc \frac{\log({\alpha_1}/{\alpha_2})}{\log\gamma} \rc=\lc \frac{\log({\alpha_2}/{\alpha_1})}{\log(1/\gamma)} \rc
~\le~\lc \frac{({ \alpha_2}/{ \alpha_1})-1}{(1- \gamma)} \rc ~=\mathrel{\mathop:}~ r
,\eeq
where the  inequality above uses the fact that  $1-(1/a)\le \log a\le a-1$ for any scalar $a>0$.
Therefore, any point $x\in E(\alpha_2)$ satisfies $G^rx\in E(\alpha_1)$.
As $E(\alpha_1)$ is invariant and $E(\alpha_1)\subseteq P$, we conclude that if $G^tx\in P$ for  $t=1,\ldots,r$, then $G^tx\in P$ for all $t\ge r$.
This establishes that $\mathcal S = {S}_r$ (in fact, we have shown that $\mathcal S = {S}_r= {S}_{\bar r}$). 
Note that the numbers $\alpha_1,\alpha_2,\gamma,$ and hence $r,$ can be computed in time polynomial in $\sigma(A,b,G).$ This completes the first part of the proof.

\textbf{Solving R-LD-LP in polynomial time.}
We next show that the number of steps to convergence is itself polynomially bounded by $\sigma(A,b,G,\rho^*)$ when $\rho(G)$ is upper bounded by a rational constant $\rho^*<1$. This would imply that after a polynomial number of steps $r$, we would have $\mathcal S = {S}_r$ and therefore the inequalities describing $\mathcal S$ could be written down in polynomial time. As any linear function $c^Tx$ can be optimized over a polyhedron in polynomial time, this would prove the second claim of the theorem.

To this end, we compute the invariant ellipsoid $E$ in Step 1 and the shrinkage factor $\gamma$ in Step 3 slightly differently.
To find an invariant ellipsoid for $G$, we now solve the  Lyapunov equation
\beq \hat G^T\hat M\hat G-\hat M=-I, \eeq
where $\hat G=(1/\hat\rho)G$ and $\hat{\rho}=(1+\rho^*)/2$. { As  $\rho(G)<\hat\rho<1$, we have $\rho(\hat G)<1$} and therefore the above equation has a unique solution.
Moreover,  the size of $\hat M$ is polynomially bounded by $\sigma(A,b,\hat G)$ and therefore by $\sigma(A,b,G,\rho^*)$. (Recall also that $\rho^*$ is a constant throughout.) As $\hat G^T\hat M\hat G-\hat M$ is negative definite, 
$$ \hat M \succ \hat G^T\hat M\hat G = \frac1{\hat\rho^2} G^T \hat M G $$ 
which readily gives the shrinkage factor $\hat\gamma =  \hat\rho^2 $.
We can now compute $\hat\alpha_1$ and $\hat\alpha_2$ using Equations \eqref{eq:alpha1} and \eqref{eq:alpha2} with $\hat M$.
Clearly the sizes of both $\hat\alpha_1$ and $\hat\alpha_2$ are   polynomially bounded b  $\sigma(A,b,\hat G)$ and therefore by $\sigma(A,b,G,\rho^*)$.

We observe with the same argument as before that with
\beq r=\lc \frac{\log({\hat\alpha_2}/{\hat\alpha_1})}{\log(1/\hat\gamma)} \rc
,\eeq
we must have $\mathcal S=S_r.$
Note that the size of $\hat\alpha_1/\hat\alpha_2$, i.e., $\log(\hat\alpha_1/\hat\alpha_2)$, is polynomially bounded by $\sigma(A,b,G,\rho^*)$. 
As $\hat{\gamma}$ is a constant, $r$ is indeed polynomially bounded by $\sigma(A,b,G,\rho^*)$.
\end{proof}

To summarize, under the assumptions of Theorem~\ref{thm.polyhedrality}, we have provided a pseudo-polynomial time algorithm\footnote{We recall that a pseudo-polynomial time algorithm is an algorithm whose running time is polynomial in the numeric value of the input (the largest integer present in the input), but not necessarily in the size of the input (the number of bits required to represent the input).} for R-LD-LP, and a polynomial algorithm for all instances where the spectral radius of $G$ is upper bounded by a constant less than one. We end this subsection by remarking that since a convex quadratic function can be minimized over a polyhedron in polynomial time \cite{KozTarKha80}, the same complexity guarantees carry over to a generalization of R-LD-LP where the linear objective function is replaced by a convex quadratic one.


\subsection{Inner approximations to $\mathcal{S}$} \label{subsec:inner.approx.LDLP}


In this subsection, we focus on the computation of a sequence of inner approximations to the feasible set $\mathcal{S}$  that produce \emph{feasible solutions} to the R-LD-LP in each step. By minimizing $c^Tx$ over these sets, one obtains \emph{upper bounds} on the optimal value of the R-LD-LP. \aaa{\st{This complements the lower bounding procedure of the previous subsection and produces guaranteed \emph{feasible solutions} to the R-LD-LP in each step of the sequence.}} (Note that points belonging to the outer approximations $S_r$ may not be feasible to the R-LD-LP, unless we wait long enough for $S_r$ to coincide with $\mathcal{S}$.) Motivated by the analysis in Section~\ref{subsec:outer.approx.LDLP}, we are interested in the remainder of this section in the setting where $\rho(G)<1$ and $P$ is a bounded polyhedron that contains the origin in its interior. Some of the statements in our lemmas below however do not need all three assumptions.


Recall the notation $S_r$ for the outer-approximating polyhedra defined in (\ref{eq:Sr.LDLP}). To find a family of inner approximations to $\mathcal{S}$, we assume that an invariant set $E\subseteq P$ (with respect to $G$) is given.
We will discuss the efficient computation of the set $E$ later (cf. Section~\ref{subsubsec:IrE.LDLP} and Section~\ref{subsubsec:invariant.ellipsoid.SDP}). 
Define $S_{-1}\mathrel{\mathop{:}}=\R^n$ and for any  integer $r\geq 0$,  let
\beqn \label{eq:Ir(E)} I_r(E) = {S}_{r-1}\cap\{x\in\R^n|\:G^rx\in E\}.\eeqn
Note that $I_0(E)=E$ by definition.
We next argue that the sets $I_r(E)$ are nested and contained in $\mathcal S$.

	

\begin{lemma} \label{lemma:inner.approx}
Let   $E\subseteq P$ be invariant with respect to $G$. 
Then, for any integer $r\ge0$, 
(i) $I_r(E) \subseteq\mathcal S$,~
(ii) $I_r(E) \subseteq I_{r+1}(E) $, 
and 
(iii) if $I_r(E)=I_{r+1}(E) $,   then $I_k(E)=I_{r}(E) $ for all $k\geq r$. 
\end{lemma}
\begin{proof}
First note that as $I_0(E)=E\subseteq P$, and $E$ is invariant, we have $I_0(E) \subseteq\mathcal S$.
When $r\ge1$, if  $x\in  I_r(E) $ then $x\in {S}_{r-1}$ and therefore $ G^tx\in  P$ for $t=0,\ldots,r-1.$ 
In addition, as the set $E\subseteq P$ is invariant,  $G^rx\in E$ implies that $G^{r+t}x\in E\subseteq P$ for $t\ge0.$ 
Consequently,  $ G^tx\in  P$ for all $t\geq 0$ and therefore  $I_r(E) \subseteq\mathcal S$.

To see that $I_r(E) \subseteq I_{r+1}(E) $, note that if $x\in I_r(E)$ then $x\in S_{r-1}$ and $G^rx\in E \subseteq \mathcal S$.
As $G^rx\in \mathcal{S}$ implies $AG^rx\leq b$, we have   $x\in S_{r}$.
Furthermore, as $E$ is invariant, if $G^rx\in E$ then $G^{r+1}x\in E$ as well.
Consequently, $x\in  I_{r+1}(E) $ as desired.

To prove that the last claim also holds, we will argue that if  $I_r(E)=I_{r+1}(E) $,   then  $I_{r+1}(E)=I_{r+2}(E)$. 
As   $I_{r+1}(E)\subseteq I_{r+2}(E) $, we need to show that   $I_{r+1}(E)\supseteq I_{r+2}(E) $.
Assume  $I_{r+1}(E)\not\supseteq I_{r+2}(E)$, and let $x\in I_{r+2}(E) \setminus I_{r+1}(E)$.
In this case, as $x\in I_{r+2}(E) $, we have $x\in S_{r+1}$ and $G^{r+2}x\in E$. 
Furthermore, as $x\not\in I_{r+1}(E) $, and $x\in S_{r}\supseteq S_{r+1}$, we also have  $G^{r+1}x\not\in E$.
Now consider the point $y=Gx$. 
Clearly, $y\in S_{r}$ as  $x\in S_{r+1}$. 
Furthermore,  we have $G^{r+1}y\in E$ as  $G^{r+2}x\in E$ and therefore, $y\in I_{r+1}(E).$
However,  $G^{r}y\not\in E$ as  $G^{r+1}x\not\in E$ and therefore  $y\not\in I_{r}(E).$
This contradicts the assumption that $I_r(E)=I_{r+1}(E) $ as $y\in I_{r+1}(E)\setminus I_{r}(E)$.
\end{proof}

We conclude that for all $r\ge0$ 
\beqn I_{r}(E) \subseteq I_{r+1}(E) \subseteq \mathcal S \subseteq {S}_{r+1}\subseteq {S}_{r}\label{eq:fooo}\eeqn
provided that  the set $E$ is invariant with respect to $G$ and $E\subseteq P$.
Also note that we did not make any assumptions on $P$ or on $G$ for the inclusion relationships in (\ref{eq:fooo}) to hold.


\newcommand{\ee}{L}

\begin{lemma} \label{lemma:inner.approx.limit}
Let $P$ be a bounded polyhedron, $\rho(G)<1,$ and $E\subseteq P$ be invariant with respect to $G$. 
Furthermore, assume that $E$  contains the origin in its interior.
Under these assumptions, if $I_r(E)=I_{r+1}(E)$,   then $\mathcal S=I_{r}(E)$. 
\end{lemma}
\begin{proof}
	Let $\ee_k\mathrel{\mathop:}=\{x\in\R^n\:|\:G^kx\in E\}$ and recall that $I_k(E) = {S}_{k-1}\cap \ee_k$ and therefore $I_k(E) \subseteq {S}_{k-1}$ for all $k>0$.
Also note that  $\lim_{k\rightarrow\infty} {S}_{k} = \mathcal S$ and as  $I_r(E)=I_{r+1}(E)$ by assumption, Lemma~\ref{lemma:inner.approx} implies that $\lim_{k\rightarrow\infty} {I}_{k} (E)=I_{r}(E)$.
Therefore, taking the limit, we obtain:
$$I_r(E) = \mathcal{S}\cap \lim_{k\rightarrow\infty} \ee_k.$$
Consequently, to prove the claim, we  need to argue that  $\lim_{k\rightarrow\infty} \ee_k\supseteq \mathcal S$.
We will actually show that  $\lim_{k\rightarrow\infty} \ee_k\supseteq P$, which is sufficient as $P\supseteq \mathcal S$.

As  $\rho(G)<1,$ following the steps in the proof of Theorem~\ref{thm.polyhedrality}, we can find a positive definite matrix $M$ such that  the ellipsoid $E(\beta)=\{x\in\R^n|~x^TMx\leq \beta\}$ is invariant under the linear dynamics $G$ for all $\beta>0$.
As $E$ is full-dimensional and contains the origin is its interior, there exists a scalar $\alpha_1>0$ such that $E\supseteq E(\alpha_1)$. 
In addition, we can compute a scalar $\alpha_2>0$ such that $E(\alpha_2) \supseteq  P$. 
Therefore, we have
$$E(\alpha_1) \subseteq  E\subseteq  P \subseteq  E(\alpha_2).$$
Furthermore, the shrinkage factor given by equation~\iref{eq:shrinkage} implies that
for some nonnegative integer $m,$ all $x\in E(\alpha_2)$ satisfy $G^kx\in E(\alpha_1)$ for all $k\ge m$.
As  $E(\alpha_1) \subseteq  E$ and $P \subseteq  E(\alpha_2),$ this implies that if $x\in P$, then $G^mx\in E$ and therefore $\ee_k\supseteq P$ for all $k\ge m$.
Consequently,  $\lim_{k\rightarrow\infty} \ee_k\supseteq P$, and $\mathcal S=I_{r}(E)$ as desired.
\end{proof}

{\cred
The proof of Lemma~\ref{lemma:inner.approx.limit} shows that for any invariant set $E\subseteq P,$ one can compute a nonnegative integer $m_E$ such that $G^{m_E}x\in E$ for all $x\in P$. This implies that $\ee_k\supseteq \mathcal S$ for all $k\ge m_E$. In addition,  Theorem \ref{thm.polyhedrality} shows that for some nonnegative integer $r$, we have $\mathcal S=S_k,$  $\forall k\ge r$. Consequently, for all $k\ge \max\{m_E,r\}$ we have $\ee_k\supseteq\mathcal S=S_k,$ which we formally state next.

\begin{corollary} \label{corollary:inner.approx.finite}
	Under the assumptions of Lemma \ref{lemma:inner.approx.limit}, there is a nonnegative integer $t$ such that $\mathcal S=I_r (E)$ for all $r\ge t$.
\end{corollary}

}


\subsubsection{Computation of $I_r(E)$}\label{subsubsec:IrE.LDLP}
The construction of the sets $I_r(E)$ requires access to an invariant set $E\subseteq P$. For an R-LD-LP with $\rho(G)<1$, an invariant set for the dynamics that is always guaranteed to exist is an ellipsoid $E=\{x \in \mathbb{R}^n|~x^TMx \leq \alpha\}$. To find the positive definite matrix $M$ (that ensures $G^TMG\preceq M$) and the positive scalar $\alpha$ (that ensures $E\subseteq P$), one can follow the methodology described in steps 1 and 2 of the proof of Theorem~\ref{thm.polyhedrality}. Note that these two steps only involve matrix inversion and basic arithmetic operations. 


With $M$ and $\alpha$ fixed, one can solve the following sequence of convex quadratic programs,


\begin{flalign} \label{eq:QP}
\underset{x\in\mathbb{R}^n}{\text{minimize}} \hspace*{1cm} & c^Tx \nonumber \\
\text{s.t.} \hspace*{1cm} &  (G^rx)^TM(G^r x) \leq \alpha,  \\ \nonumber
\ & AG^kx \leq 1, k=0,\ldots,r-1,   \nonumber
\end{flalign}
indexed by an integer\footnote{Note that when $r=0$, the final set of constraints drop out.} $r\geq 0$. The feasible sets of these optimization problems are the sets $I_r(E)$ as defined in (\ref{eq:Ir(E)}), which reside inside the feasible set $\mathcal{S}$ of our R-LD-LP. Hence, the optimal values of these convex quadratic programs are upper bounds on the optimal value of the R-LD-LP. By Lemma~\ref{lemma:inner.approx}, these upper bounds monotonically improve with $r$, and by Corollary~\ref{corollary:inner.approx.finite}, they reach the optimal value of the R-LD-LP in a finite number of steps. Although this approach is simple and convergent, it is suboptimal in terms of the quality of the upper bounds that it returns in each iteration. We explain how one can do better next.

\subsubsection{Computation of improved inner approximations}\label{subsubsec:invariant.ellipsoid.SDP}



Our improvement over the algorithm suggested in the last subsection is based on answers to the following two basic questions: (i) Instead of finding \emph{any} invariant ellipsoid $E$ and then optimizing over the sets $I_r(E)$ generated by it, can we search for an ``optimal'' invariant ellipsoid at the same time as we optimize over $I_r(E)$? (ii) Instead of working with a fixed invariant ellipsoid throughout the hierarchy, can we reoptimize the ellipsoid at each iteration? We show here that \emph{semidefinite programming} (SDP) can achieve both of these goals at once.

Let $r=0,1,\ldots$ be the index of our hierarchy. At step $r$, the strategy is to find an ellipsoid $E_r$, defined as the unit sublevel set of a quadratic form $x^TH_rx$, which satisfies the following properties:

\begin{enumerate}
	\item The set $E_r$ is invariant under the dynamics $x_{k+1}=Gx_k$.
	\item The set $E_r$ is contained in the polytope $P$. 
	\item Among all ellipsoids that have the previous two properties, $E_r$ is one that gives the minimum value of $c^Tx$ as $x$ ranges over the points in $\mathbb{R}^n$ that land in $E_r$ after $r$ steps and do not leave $P$ before doing so. (The set of such points will be denoted by $I_r(E_r)$.) 
\end{enumerate} 

As we are under the running assumption that the origin is in the interior of our polytope $P,$ the vector $b\in\mathbb{R}^m$ in the description $\{x\in\mathbb{R}^n| Ax\leq b\}$ of $P$ is elementwise positive. Hence, by rescaling, we can without loss of generality take $b$ to be the all ones vector. With this in mind, here is a mathematical description of the above optimization problem\footnote{We use the notation $S^{n\times n}$ to refer to the set of $n\times n$ real symmetric matrices.}:


%


\begin{flalign} \label{inner.ellipse.nonconvex.formulation}
\underset{x\in\mathbb{R}^n, H\in S^{n\times n}}{\text{minimize}} \hspace*{1cm} & c^Tx \nonumber \\
\text{s.t.} \hspace*{1cm} &  H\succ 0, \\ \nonumber
\ & G^THG\preceq H,  \\ \nonumber
\ & \forall z\in\mathbb{R}^n,\  z^THz\leq 1\implies Az\leq 1, \\\nonumber 
\ & (G^rx)^TH(G^rx)\leq 1,  \\ \nonumber
\ & AG^k x\leq 1, k=0,\ldots,r-1.   \nonumber
\end{flalign}
%


%
%
If the pair $(x_r,H_r)$ is an optimal solution to this problem, then we let $$E_r=\{z\in\mathbb{R}^n|~\ z^TH_rz\leq 1\},$$ 
\begin{equation}\label{eq:I_r(E_r)}
I_r(E_r)=\{z\in\mathbb{R}^n|~\ (G^rz)^TH_r(G^rz)\leq 1, AG^k z\leq 1, k=0,\ldots,r-1\},
\end{equation}
and $x_r$ will be our candidate suboptimal solution to R-LD-LP. There are two challenges to overcome with the formulation in (\ref{inner.ellipse.nonconvex.formulation}). First, the constraint $$\forall z\in\mathbb{R}^n,\  z^THz\leq 1\implies Az\leq 1$$ needs to be rewritten to remove the universal quantifier.
\st{It may seem natural to employ the $\mathcal{S}$-procedure~
for this purpose. However, even if one does this, the unknown ``multiplier'' of the $\mathcal{S}$-procedure will multiply the unknown matrix $H$, resulting in a nonconvex constraint. } 
Second, the decision variables $x$ and $H$ are multiplying each other in the constraint $(G^rx)^TH(G^rx)\leq 1$, which again makes the constraint nonconvex. Nevertheless, we show next that one can get around these issues and formulate problem (\ref{inner.ellipse.nonconvex.formulation}) \st{\emph{exactly}} \aaa{exactly} as an SDP.  The main ingredients of the proof are Schur complements, polar duality theory of convex sets (see e.g.~\cite{barvinok2002course},~\cite{rostalski2010dualities}), and duality of linear dynamical systems under transposition of the matrix $G$. 
\st{These ideas and the exact SDP formulation have first appeared in [4].}~
\aaa{Furthermore, we establish that the feasible solutions to R-LD-LP that are produced by our SDPs become optimal in a number of steps that can be computed in polynomial time.}

\begin{theorem} \label{thm:inner.SDP} Suppose $\rho(G)<1$ and the set $P=\{x\in\mathbb{R}^n|\ Ax\leq 1 \}$ is bounded. Let $a_i$ denote the transpose of the $i$-th row of the matrix $A\in\mathbb{R}^{m\times n}$ and consider the following semidefinite program:
	\begin{flalign} \label{eq:inner.ellipse.sdp}
	\underset{x\in\mathbb{R}^n, Q\in S^{n\times n}}{\text{\emph{minimize}}} \hspace*{1cm} & c^Tx \nonumber \\
	\text{\emph{s.t.}} \hspace*{1cm} &  Q\succ 0, \\ \nonumber
	\ & GQG^T\preceq Q,  \\ \nonumber
	\ & a_i^TQa_i\leq 1, i=1,\ldots,m,  \\ \nonumber
	\ &  \begin{pmatrix}
	Q &G^rx \\ (G^rx)^T &1
	\end{pmatrix}\succeq 0, \\
	\ & AG^k x\leq 1, k=0,\ldots,r-1.   \nonumber
	\end{flalign}
	Then, 
	\begin{enumerate}[(i)]
		\item the optimal values of problems (\ref{inner.ellipse.nonconvex.formulation}) and (\ref{eq:inner.ellipse.sdp}) are the same, the optimal vectors $x_r$ in the two problems are the same, and the optimal matrices $H_r$ and $Q_r$ are related via $Q_r=H_r^{-1}$. 
		\item the optimal values of the SDPs in (\ref{eq:inner.ellipse.sdp}) provide upper bounds on the optimal value of the R-LD-LP, are nonincreasing with $r$, and reach the optimal value of the R-LD-LP in a (finite) number of steps $\bar{r}$ which can be computed in time polynomial in $\sigma(A,G)$. Moreover, any optimal solution $x_r$ to the SDP in (\ref{eq:inner.ellipse.sdp}) with $r\geq \bar{r}$ is an optimal solution to R-LD-LP. 	
	\end{enumerate}
\end{theorem}

\begin{proof}
	
	(i) We show that a pair $(x,H)$ is feasible to (\ref{inner.ellipse.nonconvex.formulation}) if and only if the pair $(x,H^{-1})$ is feasible to (\ref{eq:inner.ellipse.sdp}). Indeed, we have $H \succ 0 \Leftrightarrow H^{-1}\succ 0$ as the eigenvalues of $H^{-1}$ are the inverse of the eigenvalues of $H$. Moreover, by two applications of the Schur complement (see, e.g.,~\cite[Appendix A.5.5]{BoydBook}), we observe that  
	$$ G^THG\preceq H  \Leftrightarrow \begin{pmatrix}
	H^{-1} &G \\ G^T &H
	\end{pmatrix}\succeq 0 \Leftrightarrow    GH^{-1}G^T\preceq H^{-1}.$$
	We also have that
	$$(G^rx)^TH(G^rx)\leq 1 \Leftrightarrow    \begin{pmatrix}
	H^{-1} &G^rx \\ (G^rx)^T &1
	\end{pmatrix}\succeq 0,$$
due to the Schur complement once again. Recall now that for a set $T\subseteq \mathbb{R}^n,$ its polar dual $T^\circ$ is defined as $$T^\circ\mathrel{\mathop:}=\{y\in\mathbb{R}^n|\ y^Tx\leq 1, \forall x \in T  \}.$$
	
	Let $E\mathrel{\mathop:}=\{z\in\mathbb{R}^n| \ z^THz\leq 1  \}$ and $P\mathrel{\mathop:}=\{z\in\mathbb{R}^n| \ Az\leq 1  \}$. One can verify that (i) $E\subseteq P \Leftrightarrow P^\circ\subseteq E^\circ$, (ii) $E^\circ=\{y\in\mathbb{R}^n|\ y^TH^{-1}y\leq 1 \}$, and (iii) $P^\circ=\conv\{a_1,\ldots,a_m\},$ where $conv$ here denotes the convex hull operation. Hence we have 
	$$ (\forall z\in\mathbb{R}^n, \ z^THz\leq 1\implies Az\leq 1) \Leftrightarrow   a_i^TH^{-1}a_i\leq 1, i=1,\ldots,m.$$

	(ii) The statement that the optimal value of (\ref{eq:inner.ellipse.sdp}) is an upper bound on the optimal value of the R-LD-LP follows from the fact that this SDP is constraining the optimal solution $x_r$ to be in $I_r(E_r)$, as defined in (\ref{eq:I_r(E_r)}), which is contained in $\mathcal{S}$ by construction (cf. Lemma~\ref{lemma:inner.approx}). Furthermore, if a pair $(x,Q)$ is feasible to the SDP in (\ref{eq:inner.ellipse.sdp}) at level $r$, then it is also feasible to the SDP at level $r+1$. This is because $E\mathrel{\mathop:}=\{y\in\mathbb{R}^n| \ y^TQ^{-1}y\leq 1\}\subseteq P,$ and $G^rx\in E\Rightarrow G^{r+1}x\in E$ by invariance of $E$ under $G$. Hence the claim about the monotonic improvement of the upper bounds follows. 
	
	To prove the statement about finite termination of this SDP hierarchy in a polynomially-computable number of steps, let $M\succ 0$ be the unique solution to the linear system  $G^TMG-M=-I$, $\alpha_1>0$ be as in (\ref{eq:alpha1}) with $b_1=\cdots=b_m=1$, $\alpha_2$ be as in (\ref{eq:alpha2}), $\gamma$ be as in (\ref{eq:shrinkage}), and $$\bar{r}=\lc \frac{({ \alpha_2}/{ \alpha_1})-1}{(1- \gamma)} \rc.$$ The proof of Theorem~\ref{thm.polyhedrality} already shows that this number can be computed in polynomial time. Let $\bar{x}$ be any optimal solution to the R-LD-LP. We claim that the pair $(\bar{x},\alpha_1 M^{-1})$ is a feasible solution to the SDP in (\ref{eq:inner.ellipse.sdp}) with $r=\bar{r}$. Clearly, the constraints $AG^k\bar x\leq 1, k=0,\ldots, \bar{r}-1$ are satisfied as $\bar{x}\in\mathcal{S}$. Moreover, the proof of Theorem~\ref{thm.polyhedrality} shows that the set $\bar{E}\mathrel{\mathop:}=\{y\in\mathbb{R}^n| \ y^T\frac{M}{\alpha_1}y\leq 1\}$ is contained in $P$ and is such that $G^{\bar{r}}x\in \bar{E}, \forall x\in P.$ This, together with the equation $G^TMG-M=-I$ and the fact that $\bar{x}\in P,$ implies that the pair $(\bar{x},\frac{M}{\alpha_1})$ is feasible for the problem in (\ref{inner.ellipse.nonconvex.formulation}) with $r=\bar{r}$. In view of the proof of part (i) of the current theorem, our claim about feasibility of $(\bar{x},\alpha_1 M^{-1})$ to the SDP in (\ref{eq:inner.ellipse.sdp}) with $r=\bar{r}$ follows. To finish the proof, let $(x_{\bar{r}},Q_{\bar{r}})$ be an optimal solution to this SDP. We must have $c^Tx_{\bar{r}}\leq c^T\bar{x}$ as we have just argued $\bar{x}$ is feasible to the SDP. Yet $c^Tx_{\bar{r}}\geq c^T\bar{x}$ as $I_{\bar{r}}(E_{\bar{r}})\subseteq \mathcal{S}$ and $x_{\bar{r}}\in I_{\bar{r}}(E_{\bar{r}})$. Hence, the optimal value of the SDP matches the optimal value of the R-LD-LP for all $r\geq \bar{r}$. Consequently, optimal solutions $x_r$ to the SDP must be optimal to the R-LD-LP for all $r\geq \bar{r}$ as they achieve the optimal value and belong to $\mathcal{S}$. 	
\end{proof}

We observe that the size of the semidefinite constraints in (\ref{eq:inner.ellipse.sdp}), which are the most expensive constraints in that optimization problem, does not grow with $r$. Let us now give an example.

%
%

\begin{Example}\label{ex:r-ld.dp}
	Consider an R-LD-LP defined by the following data:
	\begin{align*}
	A=\begin{pmatrix} 1 & 0\\ -1.5 & 0 \\ 0 & 1 \\ 0 & -1 \\ 1 & 1 \end{pmatrix}, \quad 
	b=\begin{pmatrix} 1\\ 1 \\ 1\\ 1 \end{pmatrix}, \quad
	c=\begin{pmatrix} -0.5 & -1 \end{pmatrix}, \quad
	G=\frac{4}{5} \begin{pmatrix} \cos(\theta) & \sin(\theta),\\ -\sin(\theta) & \cos(\theta)\end{pmatrix} \text{ where } \theta=\frac{\pi}{6}.
	\end{align*}
	In Figure \ref{fig:exRLDLP}, we plot the inner approximations $I_r$ and outer approximations $S_r$ to $\mathcal{S}$ for $r=0$ (on the left) and $r=1$ (on the right). Note that when $r=0$, $S_0$ is simply $P$. We also plot the optimal solution to the problem of minimizing $c^Tx$ over $S_0$ (resp. $I_0$) in Figure \ref{fig:RLDLPr0} and over $S_1$ (resp. $I_1$) in Figure \ref{fig:RLDLPr1}. We remark that the solutions do not coincide for $r=0$ but they do for $r=1$. Hence, our method converges in one step. This is further evidenced by the sequence of lower and upper bounds on the optimal value of the R-LD-LP given in Table \ref{tab:ulbs.RLDLP}, which shows that we have reached the exact optimal value at $r=1$. 
	
	\begin{figure}[H]
		\centering
		\begin{subfigure}{.5\textwidth}
			\centering
			\includegraphics[scale=0.25]{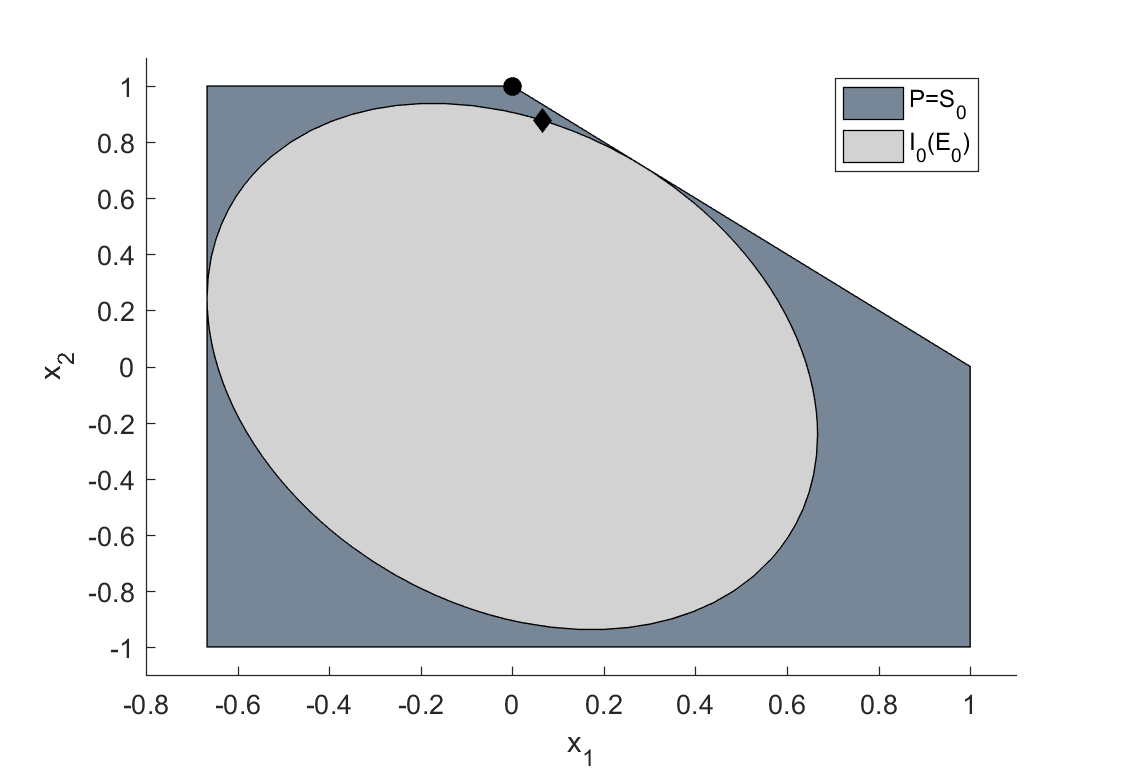}
			\caption{$r=0$}
			\label{fig:RLDLPr0}
		\end{subfigure}%
		\begin{subfigure}{.5\textwidth}
			\centering
			\includegraphics[scale=0.25]{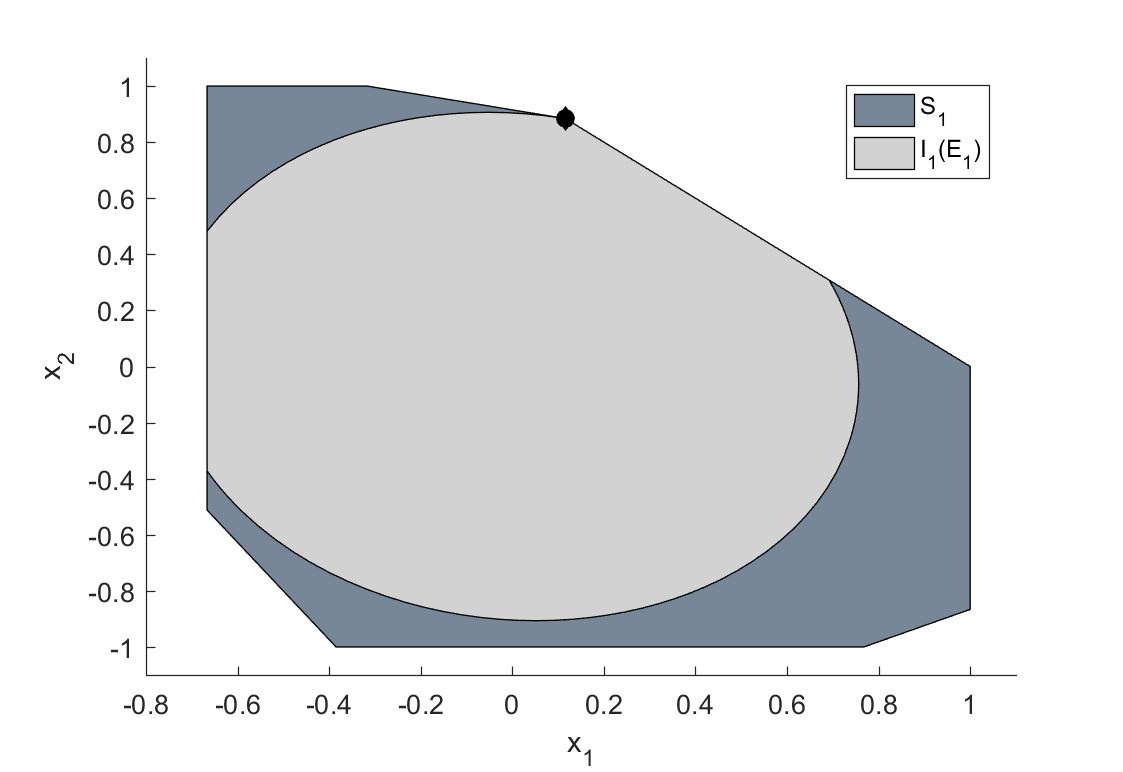}
			\caption{$r=1$}
			\label{fig:RLDLPr1}
		\end{subfigure}
		\caption{Outer and inner approximations to the feasible set of the R-LD-LP in Example~\ref{ex:r-ld.dp}.}
		\label{fig:exRLDLP}
	\end{figure}
	
	\begin{table}[H]
		\begin{center}
			\begin{tabular}{ c|c|c } 
				& $r=0$ & $r=1$ \\ 
				\hline
				Lower bounds obtained by minimizing $c^Tx$ over $S_r$ & -1 & -0.9420 \\ 
				Upper bounds obtained by minimizing $c^Tx$ over $I_r$ & -0.9105 & -0.9420 \\ 
			\end{tabular}
		\end{center}
		\caption{Our lower and upper bounds on the optimal value of the R-LD-LP in Example~\ref{ex:r-ld.dp}.}
		\label{tab:ulbs.RLDLP}
	\end{table}
	
	Figure \ref{fig:RLDLPexp} better demonstrates what the SDP is achieving at $r=1$. The set $I_1(E_1)$ in Figure~\ref{fig:RLDLPr1_2} is the set of points in $P$ that land in the set $E_1$ of Figure~\ref{fig:RLDLPr1exp} after one application of $G$. Both $E_1$ and $I_1(E_1)$ are by construction invariant inner approximations to $\mathcal{S}$. But as expected, $E_1\subseteq I_1(E_1)$, which is why $I_1(E_1)$ is the inner approximation of interest at $r=1$. Note also that the ellipsoid $E_1$ that the SDP finds at $r=1$ (Figure~\ref{fig:RLDLPr1exp}) is very different from the ellipsoid $E_0$ than the SDP finds at $r=0$ (Figure~\ref{fig:RLDLPr0}).
	
	%
	%
	%
	
	\begin{figure} 
		\centering
		\begin{subfigure}{.5\textwidth}
			\centering
			\includegraphics[scale=0.25]{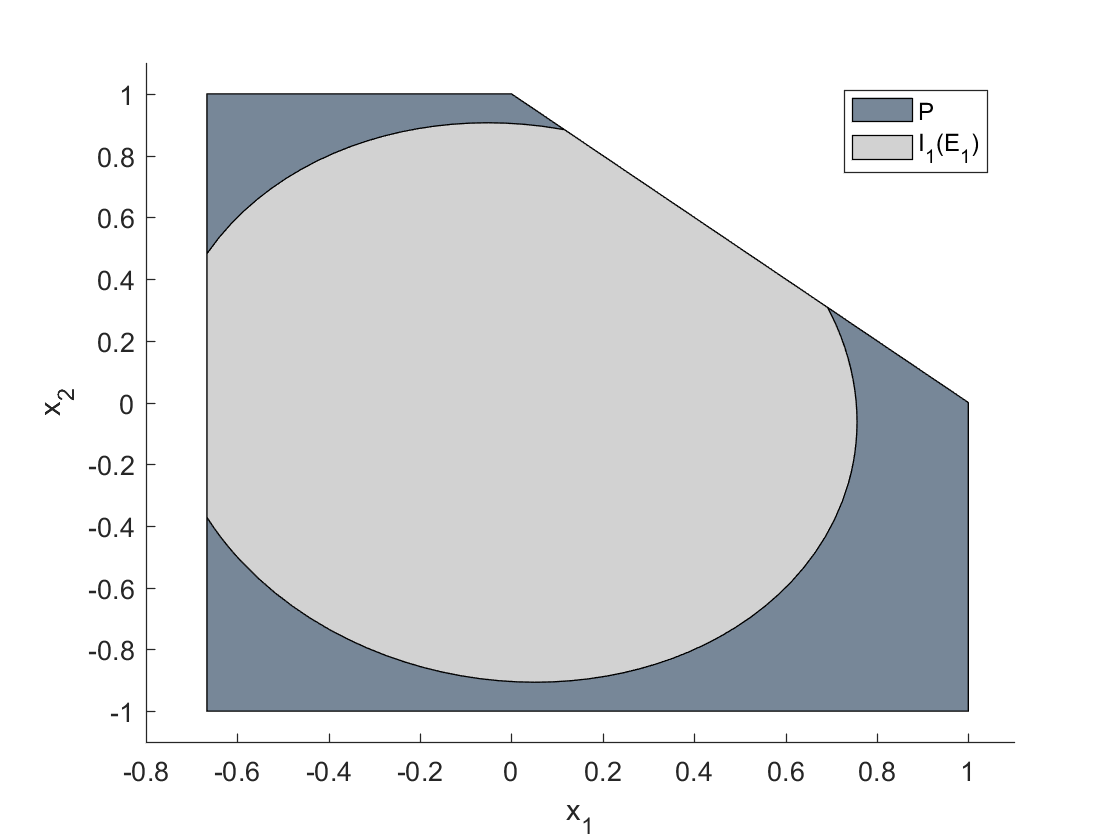}
			\caption{The set $I_1(E_1)$}
			\label{fig:RLDLPr1_2}
		\end{subfigure}%
		\begin{subfigure}{.5\textwidth}
			\centering
			\includegraphics[scale=0.25]{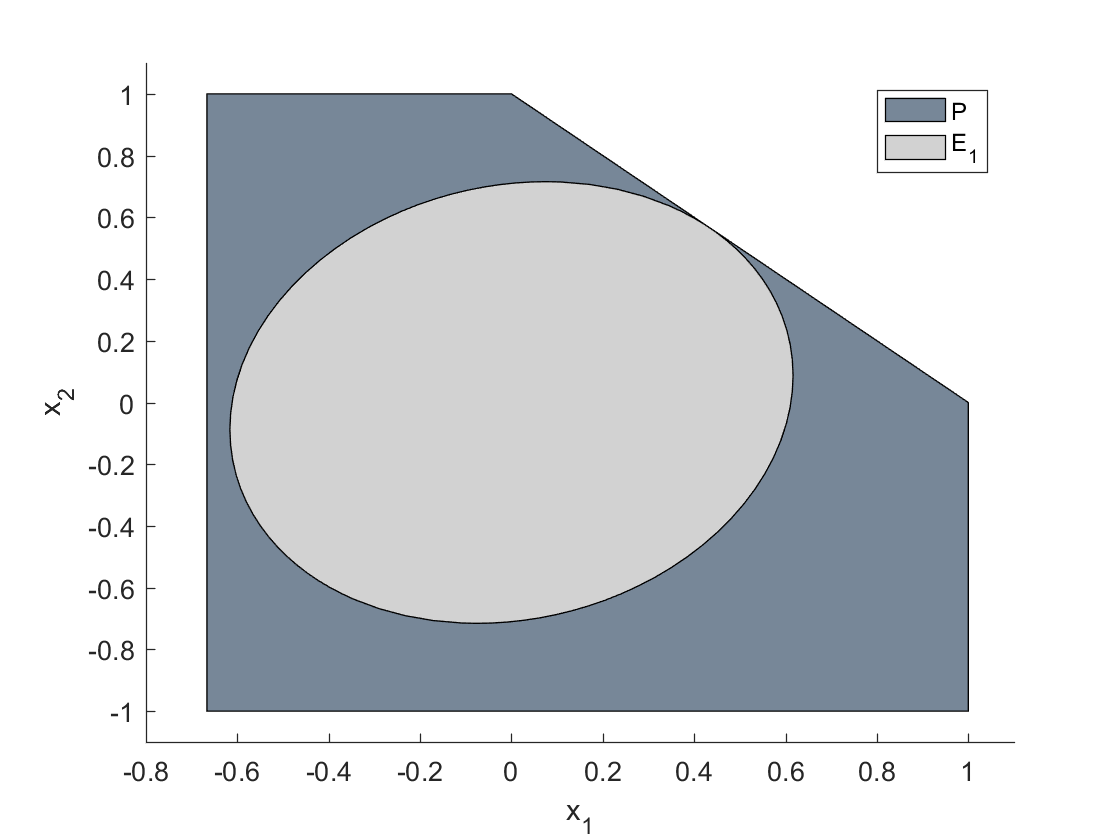}
			\caption{The set $E_1$}
			\label{fig:RLDLPr1exp}
		\end{subfigure}
		\caption{The set $I_1(E_1)$ is the set of points in $P$ that land in the ellipsoid $E_1$ after one application of $G$.}
		\label{fig:RLDLPexp}
	\end{figure}
	
\end{Example}

\section{Robust to uncertain \aaa{and time-varying} linear dynamics linear programming}\label{sec:R-ULD-LP}

In the theory of robust control, there has long been an interest in analyzing the behavior of dynamical systems whose parameters are not exactly known and can vary in time. This is motivated by the fact that in many practical applications, the physical or social dynamics of interest are hard to model exactly and are subject to external disturbances that vary with time. \aaa{\st{(Consider, e.g., the dynamics of  the spread of an epidemic, where the frequency of social interactions changes depending on the time of day, or that of a chemical reaction, whose behavior varies depending on the temperature of the environment.)}}

We consider one of the most widely-studied linear models that captures both parameter uncertainty and dependence on time (see, e.g.,~\cite{Raphael_Book},~\cite{shorten2007stability} and references therein). In this model, one is given $s$ real $n\times n$ matrices $G_1,\ldots,G_s $ and assumes that the true linear dynamics are given by a matrix in their convex hull $\conv\{G_1,\ldots,G_s\}$. (This is a polytope in the space of $s\times s$ matrices whose extreme points are given.) Moreover, at each iteration, a different matrix from this polytope can govern the dynamics. This leads to the following uncertain and time-varying dynamical system
\begin{equation}\label{eq:difference.inclusion}
x_{k+1} \in \conv\{G_1,\ldots,G_s\}x_k,
\end{equation}
where $\conv\left\{G_1,\ldots,G_s\}x\mathrel{\mathop:}=\{\sum_{i=1}^s\lambda_i G_i x|~ \lambda_i\geq 0, \sum_{i=1}^s \lambda_i=1 \right\}$. 
\aaa{Note that this model encompasses 
a nonlinear, time-invariant system $$x_{k+1}=g(x_k),$$ where $g:\mathbb{R}^n\rightarrow\mathbb{R}^n$ satisfies \begin{equation}\label{eq:uncertain.nl}
    \forall x\in\Omega, g(x)\in\ \conv\{G_1x,\ldots,G_sx\}
\end{equation} 
for the relevant subset $\Omega$ of $\mathbb{R}^n$ (see~\eqref{eq:opt.input}). Moreover, this model captures uncertainty in the nonlinear dynamics as well, as the dynamical system can be governed by \emph{any} map $g$ that satisfies~\eqref{eq:uncertain.nl}.}


In this section, we are interested in studying linear programs that must remain robust against such a dynamical system. More precisely, a \emph{robust to uncertain \aaa{and time-varying} linear dynamics linear program} (\aaa{R-UTVLD-LP}) is an optimization problem of the form
\begin{align}\label{eq:r.uld.lp}
\min_{x_0 \in \mathbb{R}^n} \left\{c^Tx_0:  x_k\in P \ \mbox{for}\ k=0,1,2,\ldots, \mbox{u.t.d.} \ x_{k+1} \in \conv\{G_1,\ldots,G_s\}x_k \right\},
\end{align}
where $P=\{x \in \mathbb{R}^n |~ Ax\leq b\}$ is a given polyhedron. The input to this problem is fully defined by $$A\in\mathbb{R}^{m\times n}, b\in\mathbb{R}^{m}, c\in \mathbb{R}^n, G_1,\ldots, G_s\in\mathbb{R}^{n\times n}.$$ It is not hard to see that an \aaa{R-UTVLD-LP} can be equivalently formulated as 
\begin{equation}\nonumber
\min_{x_0 \in \mathbb{R}^n} \left\{c^Tx_0:  x_k\in P \ \mbox{for}\ k=0,1,2,\ldots, \mbox{u.t.d.} \ x_{k+1} \in \{G_1x_k,\ldots,G_sx_k\} \right\}.
\end{equation}
Indeed, it is straightforward to check that for any integer $k\geq 1$, a point $x\in P$ leaves $P$ by some product of length $k$ out of the matrices in $\conv\{G_1,\ldots,G_s\},$ if and only if it leaves $P$ by some product of length $k$ out of the matrices in $\{G_1,\ldots,G_s\}$. 

Let $\mathcal{G}\mathrel{\mathop:}=\{G_1,\ldots,G_s\}$ and let $\mathcal{G}^k$ denote the set of all $s^k$ matrix products of length $k$ (with $\mathcal{G}^0$ consisting only of the identity matrix by convention). Let $$\mathcal{G}^*=\cup_{k=0}^\infty \mathcal{G}^k$$ be the set of all finite products from $\mathcal{G}$. An \aaa{R-UTVLD-LP} can then be reformulated as the following linear program with a countably infinite number of constraints:
\begin{equation}\label{eq:RULDLP}
\min_{x\in\R^n}\{c^Tx: Gx\in P, \forall G\in\mathcal{G}^*\}.
\end{equation}
Note that an R-LD-LP is a special case of an \aaa{R-UTVLD-LP} with $s=1$. Throughout this section, we denote the feasible set of an \aaa{R-UTVLD-LP} by 
\begin{equation}\label{eq:feas.set.uld.lp}
\mathcal{S}\mathrel{\mathop{:}}= \bigcap_{k=0}^{\infty} \{x \in \mathbb{R}^n| ~AGx \leq b, \forall G\in\mathcal{G}^k\}.
\end{equation}

Clearly, the statement of Theorem~\ref{thm:S.closed.convex.invariant.non.polyhedral.nphard} still applies to this set. Indeed, $\mathcal{S}$ is closed and convex as an infinite intersection of closed convex sets, and, by definition, invariant under multiplication by $G_1,\ldots, G_s$. Moreover, $\mathcal{S}$ is not always polyhedral even when $s=1$, and testing membership of a given point to $\mathcal{S}$ is NP-hard already when $s=1$. 
%
%
Our goal here will be to study tractable outer and inner approximations to $\mathcal{S},$ and to extend some of the statements we proved for R-LD-LPs to this more intricate setting.


\subsection{Outer approximations to $\mathcal{S}$}\label{Outer.approx.ULD}

Let
\begin{equation}\label{eq:Sr.switched}
S_r\mathrel{\mathop{:}}= \bigcap_{k=0}^{r} \{x \in \mathbb{R}^n| ~AGx \leq b, \forall G\in\mathcal{G}^k\}
\end{equation}
denote the set of points that remain in $P$ under all matrix products of length up to $r$. It is clear that these sets provide polyhedral outer approximations to $\mathcal{S}$:
$$\mathcal{S}\subseteq \ldots S_{r+1}\subseteq S_r\subseteq \ldots\subseteq S_2\subseteq S_1\subseteq S_0=P.$$ Hence, by solving LPs that minimize $c^Tx$ over $S_r,$ we obtain a nondecreasing and convergent sequence of lower bounds on the optimal value of an \aaa{R-UTVLD-LP}. We leave it to the reader to check that the statement of Lemma~\ref{lemma:termination} still holds with an almost identical proof. This gives us a way of checking finite termination of convergence of the sets $S_r$ to $\mathcal{S}.$ We now need to generalize the notion of the spectral radius to several matrices.


\begin{definition}[Rota and Strang~\cite{RoSt60}]\label{def:jsr}
Given a set of $n \times n$ matrices $\mathcal{G}=\{G_1,\ldots,G_s\}$, their \emph{joint spectral radius} (JSR) is defined as\footnote{The JSR is independent of the norm used in this definition.} 
$$\rho(\mathcal{G})\mathrel{\mathop{:}}=\lim_{k \rightarrow \infty} \max_{\sigma \in \{1,\ldots,s\}^k} ||G_{\sigma_1}\ldots G_{\sigma_k}||^{1/k}.$$
\end{definition}

The JSR characterizes the maximum growth rate that can be obtained by taking long products out of the matrices $G_1,\ldots, G_s$ in arbitrary order. Note that when $s=1$, it coincides with the spectral radius. This can be seen e.g., via the Gelfand's formula for the spectral radius. 



We observe that the statements of Propositions~\ref{prop:rho>1.convergence.not.finite},~\ref{prop:convergence.not.finite.without.origin}, and~\ref{prop:P.unbounded.not.finite} are still valid (with $\rho(G)$ replaced with $\rho(\mathcal{G})$), as they even apply to the special case of an \aaa{R-UTVLD-LP} with $s=1$. These propositions, together with the construction in the proof of part (ii) of Theorem~\ref{thm:S.closed.convex.invariant.non.polyhedral.nphard}, demonstrate that none of the three assumptions in the following theorem can be removed.


\begin{theorem}\label{thm:finite.switched}
Let $\mathcal{G}=\{G_1,\ldots,G_s\}.$ If $\rho(\mathcal{G})<1,$ $P$ is bounded, and the origin is in the interior of $P$, then $\mathcal S=S_r$ for some integer $r\geq 0.$
\end{theorem}

\begin{proof}
Let $\hat{\rho}=\frac{\rho(\mathcal{G})+1}{2}<1$. It follows (see, e.g.,~\cite{RoSt60},~\cite[Lemma II]{berger1992bounded}) that there exists a norm $f:\mathbb{R}^n\rightarrow\mathbb{R}$ such that for any $\alpha \geq 0$ and any $x\in\mathbb{R}^n,$ $f(x)\leq\alpha \Rightarrow f(G_ix)\leq \alpha \cdot \hat{\rho}$, $\forall i\in\{1,\ldots,s\}$. As $P$ contains the origin in its interior and is bounded, there exists $\alpha_2>\alpha_1>0$ such that $$\{x \in \mathbb{R}^n|~f(x) \leq \alpha_1\} \subseteq P \subseteq \{x \in \mathbb{R}^n|~f(x)\leq \alpha_2\}.$$ Hence, any point in $P$, once multiplied by any matrix product of length $r=\frac{\log(\alpha_1/\alpha_2)}{\log(\hat{\rho})},$ lands in the set $\{x \in \mathbb{R}^n|~f(x)\leq\alpha_1\}.$ As $\{x \in \mathbb{R}^n|~f(x)\leq~\alpha_1\} \subseteq \mathcal{S}$, the result follows.
\end{proof}

We remark that our proof above did not use the fact that $P$ was a polytope and would hold if $P$ were instead any compact set. The reason this proof was noticeably simpler than that of Theorem~\ref{thm.polyhedrality} is that we did not analyze how large $r$ can be. We did not do so because of two reasons: (i) the sublevel sets of the norm $f$ in the above proof may not be simple sets like ellipsoids
 that are amenable to algorithmic analysis, and (ii) even if $r$ is small, the number of inequalities describing the set $S_r$ can be as large as $\sum_{k=0}^r ms^k$, a quantity which grows very quickly when $s\geq 2$. We empirically observe, however, that the first few levels of this hierarchy often provide high-quality lower bounds on the optimal value of an \aaa{R-UTVLD-LP}. We can check this by computing upper bounds on the optimal value via a procedure that we describe in the next subsection.

Theorem~\ref{thm:finite.switched} as well as some of the theorems in the remainder of this section require the assumption that $\rho(\mathcal{G})<1$. While algorithmic decidability of this condition is currently unknown~\cite{blondel2000survey}, there is a large body of literature on the computation of (arbitrarily tight) upper bounds on the JSR, which can be utilized to verify this assumption; see e.g.~\cite{Raphael_Book},~\cite{parrilo2008approximation},~\cite{blondel2005computationally},~\cite{ahmadi2016lower} and references therein. In fact, we present a hierarchy of SDP-based sufficient conditions for checking this assumption in the next subsection (see Theorem~\ref{thm:jsr.max.of.quadratics}), which happens to also be useful for finding inner approximations to the feasible set of an \aaa{R-UTVLD-LP}.


\subsection{Inner approximations to $\mathcal{S}$}\label{Inner.approx.ULD}

In this subsection, we generalize the results of Section~\ref{subsec:inner.approx.LDLP} to the case of \aaa{R-UTVLD-LPs}. Recall our notation $\mathcal{S}$ from (\ref{eq:feas.set.uld.lp}) for the feasible set of an \aaa{R-UTVLD-LP}, and let us keep our notation $P,\mathcal{G}^k,$ and $S_r$ from the previous subsection. Let $E\subseteq P$ be any convex set that contains the origin in its interior and is invariant under multiplication by $G_1,\ldots,G_s$. Since $E$ is convex, it must also be invariant under the dynamics in (\ref{eq:difference.inclusion}). Define $S_{-1}\mathrel{\mathop{:}}=\R^n$ and for any integer $r\geq 0$, let 
\beqn \label{eq:Ir(E).switched} I_r(E) = {S}_{r-1}\cap\{x\in\R^n\:|\:Gx\in E, \forall G\in\mathcal{G}^r\}.\eeqn
Note that $I_0(E)=E$ by definition.
With this notation, the reader can verify that Lemma~\ref{lemma:inner.approx}, Lemma~\ref{lemma:inner.approx.limit}, and Corollary~\ref{corollary:inner.approx.finite} extend, with almost identical proofs, to the case where the single matrix $G$ is replaced by the set of matrices $\mathcal{G}=\{G_1,\ldots,G_s\}$. We summarize these results in the next lemma.
\begin{lemma}\label{lem:I_rE.properties.switched}
Let $E\subseteq P$ be convex\footnote{We ask that $E$ be convex, so its invariance with respect to $\mathcal{G}$ would imply its invariance with respect to the matrices in $conv(\mathcal{G})$. It is easy to see that in general, if a set $T$ is invariant under $\mathcal{G}$, then $conv(T)$ is invariant under $conv(\mathcal{G})$.} and invariant with respect to $\mathcal{G}=\{G_1,\ldots, G_s\}$. The sets $I_r(E)$ in (\ref{eq:Ir(E).switched}) satisfy the following properties:
$$I_r(E)\subseteq \mathcal{S}, \mbox{ and } I_r(E)\subseteq I_{r+1}(E) \mbox{ for all } r\geq 0.$$
Moreover, if $P$ is bounded, $\rho(\mathcal{G})<1,$ and $E$ contains the origin in its interior, then there exists a nonnegative integer $t$ such that $\mathcal{S}=I_r(E)$ for all $r\geq t$.
\end{lemma} 

In words, Lemma~\ref{lem:I_rE.properties.switched} states that the sets $I_r(E)$ provide an improving sequence of inner approximations to $\mathcal{S}$ and coincide with $\mathcal{S}$ in finite time.

\subsubsection{Computation of $I_r(E)$}\label{subsubsec:IrE.switched}
The construction of the sets $I_r(E)$ requires access to a convex  invariant set $E\subseteq P$.
A nontrivial challenge here is that unlike the case of a single matrix (Section~\ref{subsubsec:IrE.LDLP}), it is possible to have $\rho(\mathcal{G})<1$ and yet not have an ellipsoid that is invariant under the action of the matrices $G_1,\ldots,G_s.$ For example, the matrices
\begin{equation}\nonumber
G_{1}=\gamma\left[
\begin{array}
[c]{cc}%
1 & 0\\
1 & 0
\end{array}
\right]  ,\text{ }G_{2}=\gamma\left[
\begin{array}
[c]{cr}%
0 & 1\\
0 & -1
\end{array}
\right]
\end{equation}
have JSR less than one for $\gamma\in[0,1),$ but only admit a common invariant ellipsoid for $\gamma\in~[0,\frac{1}{\sqrt{2}}]$~\cite{ando1998simultaneous}. It turns out however, that if the JSR is less than one, then there is always an invariant set which is the intersection of a finite number of ellipsoids. Moreover, these ellipsoids can be found via semidefinite programming.

\begin{theorem}[see Theorem 6.1 and Theorem 2.4 of~\cite{JSR_path.complete_journal}]\label{thm:jsr.max.of.quadratics}
Let $\mathcal{G}=\{G_1,\ldots,G_s\}$ be a set of $n\times n$ matrices. Then, for any integer $l\geq 1$, if $\rho(\mathcal{G})\leq \frac{1}{\sqrt[2l]{n}}$, there exist $s^{l-1}$ real symmetric matrices $H_{\pi}$, where $\pi \in\{1,\ldots,s\}^{l-1}$ is a multi-index, such that
\begin{equation}\label{eq:converse.thm.LMIs}
\begin{array}{rll}
H_{\pi}&\succ& 0\ \ \ \  \forall \pi \in\{1,\ldots,s\}^{l-1},\\
G_j^TH_{i \sigma}G_j&\preceq&H_{\sigma j}, ~\forall \sigma \in\{1,\ldots,s\}^{l-2},~ \forall i,j \in\{1,\ldots,s\}.
\end{array}
\end{equation}
Conversely, existence of a set of symmetric matrices $H_\pi$ that satisfy the semidefinite constraints in (\ref{eq:converse.thm.LMIs}) strictly\footnote{If $\rho(\mathcal{G})< \frac{1}{\sqrt[2l]{n}}$, the constraints in (\ref{eq:converse.thm.LMIs}) will indeed be strictly feasible as one can apply the first part of this theorem to $\beta\mathcal{G}\mathrel{\mathop:}=\{\beta G_1,\ldots,\beta G_s  \}$ for $\beta>1$ and small enough. (Note that the JSR is a continuous function of the entries of $\mathcal{G}$~\cite{Raphael_Book} and satisfies the homogeneity relation $\rho(\beta\mathcal{G})=\beta\rho(\mathcal{G})$.)} implies that $\rho(\mathcal{G})<1.$ Moreover, if (\ref{eq:converse.thm.LMIs}) is satisfied, then for any scalar $\alpha\geq 0$, the set 
\begin{equation}\label{eq:def.F.alpha}
\mathcal{F}_\alpha\mathrel{\mathop:}=\big\{x\in\mathbb{R}^n|~ x^T H_{\pi} x\leq \alpha, \forall  \pi \in\{1,\ldots,s\}^{l-1}    \big\} 
\end{equation}
is invariant under multiplication by $G_1,\ldots,G_s$.
\end{theorem}
\begin{remark}\label{rmk:notation}
By convention, when $l=1$, the decision variable in (\ref{eq:converse.thm.LMIs}) is just a single matrix $H$ and the constraints in (\ref{eq:converse.thm.LMIs}) should read $$H\succ 0, \quad  G_j^THG_j \preceq H, \forall j\in \{1,\ldots,s\}.$$ In the case where $l=2$, one should solve (\ref{eq:converse.thm.LMIs}) with the convention that $\{1,\ldots,s\}^0$ is the empty set. This means that the decision variables are $H_1,\ldots, H_s$ and the constraints are 
$$H_1\succ 0,\ldots,H_s\succ 0, \quad  G_j^TH_iG_j \preceq H_j, \forall i,j\in \{1,\ldots,s\}.$$
\end{remark}

\begin{proof}[Proof of Theorem~\ref{thm:jsr.max.of.quadratics}]
The proof of this theorem appears in \cite{JSR_path.complete_journal}, except for the part about invariance of the sets $\mathcal{F}_{\alpha},$ which we include here for completeness. We need to show that the constraints in (\ref{eq:converse.thm.LMIs}) imply $$x \in \mathcal{F}_{\alpha} \Rightarrow G_jx \in \mathcal{F}_{\alpha}, \forall j=1,\ldots,s.$$ Let $\bar{x} \in \mathcal{F}_{\alpha}$ and define a function $W:\mathbb{R}^n\rightarrow\mathbb{R}$ as $$W(x)\mathrel{\mathop{:}}=\max_{\pi \in \{1,\ldots,s\}^{l-1}} \{x^TH_\pi x\}.$$ 
By definition of $\mathcal{F}_{\alpha}$, $W(\bar{x}) \leq \alpha$. Furthermore, from the second set of inequalities in (\ref{eq:converse.thm.LMIs}), it is easy to see that $W(G_jx) \leq W(x), \forall j=1,\ldots,s$ and $x \in \mathbb{R}^n.$ Indeed, (\ref{eq:converse.thm.LMIs}) implies that $\forall \sigma \in \{1,\ldots,s\}^{l-2}$, $\forall i, j \in \{1,\ldots,s\}$ and $\forall x \in \mathbb{R}^n$, 
$$x^TG_j^TH_{i\sigma}G_jx\leq \max_{\hat{\sigma} \in \{1,\ldots,s\}^{l-2}, \hat{j} \in \{1,\ldots,s\}} x^TH_{\hat{\sigma} \hat{j}}x=W(x).$$
We hence deduce that $W(G_j \bar{x}) \leq W(\bar{x})\leq \alpha,$ for $j=1,\ldots,s,$ and so $G_j\bar{x} \in \mathcal{F}_{\alpha}$ for $j=1,\ldots,s.$
\end{proof}

Going back to the computation of the convex invariant set $E \subseteq P$, which is needed for the construction of the inner approximations $I_r(E)$ in (\ref{eq:Ir(E).switched}), we first find the smallest integer $l \geq 1$ for which the SDP in (\ref{eq:converse.thm.LMIs}) is feasible. (Note that we never need to compute the JSR.) Once this is done, for any fixed $\alpha\geq 0$, the set $\mathcal{F}_{\alpha}$ in (\ref{eq:def.F.alpha}) provides us with a convex and invariant set. We now need to find a small enough  $\bar{\alpha}>0$ such that $\mathcal{F}_{\bar{\alpha}} \subseteq P$. A simple way of doing this is to require that one ellipsoid, say the first, be in the polytope. With this approach, $\bar{\alpha}$ can be computed by following the procedure described in Step 2 of the proof of Theorem \ref{thm.polyhedrality}, which only requires matrix inversion. 

With $\bar{\alpha}$ and the matrices $\{H_{\pi}\}$ fixed, consider the following sequence of convex quadratic programs,
\begin{flalign} \label{eq:QP.switched}
\underset{x\in\mathbb{R}^n}{\text{minimize}} \hspace*{1cm} & c^Tx \nonumber \\
\text{s.t.} \hspace*{1cm} &  (Gx)^TH_{\pi}(G x) \leq \bar{\alpha}, \forall G \in \mathcal{G}^r, \forall \pi \in \{1,\ldots,s\}^{l-1}  \\ \nonumber
\ & AGx \leq 1, \forall G \in \bigcup_{k=0}^{r-1} \mathcal{G}^k,   \nonumber
\end{flalign}
indexed by an integer $r\geq 0$.
The feasible sets of these optimization problems are the sets $I_r(E)$ as defined in (\ref{eq:Ir(E).switched}) with $E=\mathcal{F}_{\bar{\alpha}}$. As $I_r(E)\subseteq \mathcal{S}$ for all $r\geq 0$, the optimal values of these convex quadratic programs are upper bounds on the optimal value of the \aaa{R-UTVLD-LP}. Lemma~\ref{lem:I_rE.properties.switched} further implies that these upper bounds monotonically improve with $r$ and reach the optimal value of the \aaa{R-UTVLD-LP} in a finite number of steps. While this approach already achieves finite convergence, there is much room for improvement as the invariant set $E$ is fixed throughout the iterations and is designed without taking into consideration the objective function.

\subsubsection{Computation of improved inner approximations}\label{subsubsec:IrEr.switched} 

Our goal now is to find invariant sets $E_r$ that result in the sets $I_r(E_r)$ in (\ref{eq:Ir(E).switched}) that best approximate the feasible $\mathcal{S}$ of an \aaa{R-UTVLD-LP} in the direction of its objective function. To do this, we first find the smallest integer $l$ for which the SDP in (\ref{eq:converse.thm.LMIs}) is feasible. We fix this number $l$ throughout. Our sets $E_r$, for $r=0,1,\ldots$, will then be given by 
\begin{align}\label{eq:def.Er.switched}
E_r=\big\{z\in\mathbb{R}^n|~ z^T H_{\pi ,r} z\leq 1, \forall  \pi \in\{1,\ldots,s\}^{l-1}    \big\}, 
\end{align}
where the symmetric matrices $H_{\pi,r}$ are optimal solutions to the following optimization problem: 

\begin{flalign} 
\underset{x\in\mathbb{R}^n, H_{\pi} \in S^{n\times n}}{\text{minimize}} \hspace{5mm} & c^Tx \label{eq:nncvx.opt.switched} \\
\text{s.t.} \hspace*{1cm} & H_{\pi }\succ 0, \forall \pi\in\{1,\ldots,s\}^{l-1},\nonumber \\
\ & G_j^TH_{i \sigma}G_j \preceq H_{\sigma j},
 \forall  \sigma \in\{1,\ldots,s\}^{l-2}, \forall i, j \in\{1,\ldots,s\}, \label{eq:nncvx.opt.switched2} \\
\ & \forall z\in \mathbb{R}^n,~ z^TH_{1\ldots 1}z\leq 1\Rightarrow Az\leq 1, \label{eq:nncvx.opt.switched3} \\
\ & (Gx)^TH_{\pi}(Gx)\leq 1, \forall \pi \in\{1,\ldots,s\}^{l-1}, \forall G \in \mathcal{G}^r, \label{eq:nncvx.opt.switched4} \\ 
\ & AG x\leq 1, \forall G \in \bigcup_{k=0}^{r-1} \mathcal{G} ^k.  \label{eq:nncvx.opt.switched5}
\end{flalign}
Our Remark~\ref{rmk:notation} regarding notation still applies here. Note that constraint (\ref{eq:nncvx.opt.switched2}) imposes that the set $E_r$ in (\ref{eq:def.Er.switched}) be invariant under the dynamics in (\ref{eq:difference.inclusion}). Constraint (\ref{eq:nncvx.opt.switched3}) forces one of the ellipsoids to be within the polytope, which implies that the intersection $E_r$ of all ellipsoids will be in the polytope (this is obviously only a sufficient condition for $E_r \subseteq P$). We remark here that choosing $H_{1\ldots 1}$ to feature in this constraint is without loss of generality; as $H_{1\ldots 1}$ is a variable of the problem, the optimization problem will naturally pick the ``best'' ellipsoid to constrain to be in the polytope. Constraints (\ref{eq:nncvx.opt.switched4}) and (\ref{eq:nncvx.opt.switched5}) force the point $x$ to land in $E_r$ under all products of length $r$ without leaving $P$ before time $r$.

Once this optimization problem is solved to obtain an optimal solution $x_r$ and $\{H_{\pi ,r}\}$, our inner approximation to $\mathcal{S}$ at step $r$ will be the set
\begin{equation}\label{eq:I_r(E_r).switched}
\begin{aligned}
I_r(E_r)=\Big\{z\in\mathbb{R}^n|~\ (Gz)^TH_{\pi,r}(Gz)\leq 1, \forall \pi \in \{1,\ldots,s\}^{l-1}, \forall G \in \mathcal{G}^r, 
AG z\leq 1, \forall G \in \bigcup_{k=0}^{r-1} \mathcal{G}^k\Big\},
\end{aligned}
\end{equation}
and $x_r$ will serve as our candidate suboptimal solution to \aaa{R-UTVLD-LP}. Just as we did in Section~\ref{subsubsec:invariant.ellipsoid.SDP}, we next show that by a reparameterization, the above optimization problem can be cast as an SDP.

\begin{theorem}\label{thm:inner.SDP.switched}
Suppose $\rho(\mathcal{G})<1$ and the set $P=\{x\in\mathbb{R}^n|\ Ax\leq 1 \}$ is bounded. Let $a_i$ denote the transpose of the $i$-th row of the matrix $A\in\mathbb{R}^{m\times n}$ and consider the following semidefinite program:
	\begin{flalign} \label{eq:inner.ellipse.sdp.switched}
	\underset{x\in\mathbb{R}^n, Q_{\pi} \in S^{n\times n}}{\text{\emph{minimize}}} \hspace*{1cm} & c^Tx \nonumber \\
	\text{\emph{s.t.}} \hspace*{1cm} &  Q_{\pi}\succ 0, \forall \pi \in\{1,\ldots,s\}^{l-1}, \\ \nonumber
 &	G_jQ_{\sigma j}G_j^T \preceq Q_{i \sigma}, \forall \sigma \in\{1,\ldots,s\}^{l-2}, \forall i,j\in\{1,\ldots,s\}, \nonumber \\
	\ & a_j^TQ_{1\ldots 1}a_j\leq 1, \forall j\in \{1,\ldots,m\} \nonumber \\ \nonumber
	\ &  \begin{pmatrix}
	Q_{\pi } &Gx \\ (Gx)^T &1
	\end{pmatrix}\succeq 0, \forall G \in \mathcal{G}^r, \forall \pi \in\{1,\ldots,s\}^{l-1},\\
	\ & AG x\leq 1, \forall G \in \bigcup_{k=0}^{r-1} \mathcal{G}^k .   \nonumber
	\end{flalign}
	Then, 
	\begin{enumerate}[(i)]
		\item the optimal values of problems (\ref{eq:nncvx.opt.switched}) and (\ref{eq:inner.ellipse.sdp.switched}) are the same, the optimal vectors $x_r$ in the two problems are the same, and the optimal matrices $H_{\pi, r}$ and $Q_{\pi, r}$ are related via $Q_{\pi, r}=H_{\pi, r}^{-1}$. 
		\item the optimal values of the SDPs in (\ref{eq:inner.ellipse.sdp.switched}) provide upper bounds on the optimal value of the \aaa{R-UTVLD-LP}, are nonincreasing with $r$, and reach the optimal value of the \aaa{R-UTVLD-LP} in a finite number of steps $\bar{r}$. Moreover, any optimal solution $x_r$ to the SDP in (\ref{eq:inner.ellipse.sdp.switched}) with $r\geq \bar{r}$ is an optimal solution to \aaa{R-UTVLD-LP}. 	
	\end{enumerate}
\end{theorem}

\begin{proof} 
The proof of part (i) uses the same exact ideas as the proof of part (i) of Theorem~\ref{thm:inner.SDP} (Schur complements and polar duality of polytopes and ellipsoids) and is left to the reader. In particular, this proof would use Schur complements to show that $$G_j^TH_{i \sigma}G_j \preceq H_{\sigma j} \iff G_j H^{-1}_{\sigma j}G_j^T \preceq H^{-1}_{i \sigma}.$$
We now prove part (ii). The statement that the optimal value of (\ref{eq:inner.ellipse.sdp.switched}) is an upper bound on the optimal value of the R-LD-LP follows from the fact that in view of part (i), the last two sets of constraints of this SDP are constraining the optimal solution $x_r$ to be in $I_r(E_r)$, as defined in (\ref{eq:I_r(E_r).switched}). We know that $I_r(E_r)\subseteq\mathcal{S}, \forall r\geq 0$ as points in $I_r(E)$ land in the invariant set $E_r\subseteq P$ (cf. (\ref{eq:def.Er.switched})) in $r$ steps without leaving $P$ before time $r$. To see the claim about the monotonic improvement of our upper bounds, observe that if $x,\{Q_{\pi} \}$ are feasible to the SDP in (\ref{eq:inner.ellipse.sdp.switched}) at level $r$, then they are also feasible to the SDP at level $r+1$. This is because we have the inclusion $$E\mathrel{\mathop:}=\big\{z\in\mathbb{R}^n|~ z^T Q^{-1}_{\pi} z\leq 1, \forall  \pi \in\{1,\ldots,s\}^{l-1}    \big\} \subseteq P,$$ by the third set of constraints in (\ref{eq:inner.ellipse.sdp.switched}), and the implication $$Gx\in E,\forall G\in\mathcal{G}^r\Rightarrow Gx\in E, \forall G\in\mathcal{G}^{r+1}$$ by invariance of $E$ under $\{G_1,\ldots,G_s\}$ as enforced by the second set of constraints in (\ref{eq:inner.ellipse.sdp.switched}).


We now show that there exists an integer $\bar{r}\geq 0$ such that the optimal value $c_r$ to the SDP at level $r$ is equal to the optimal value $c^*$ of the \aaa{R-UTVLD-LP} for all $r \geq \bar{r}.$ Let $E=E_0$ as defined in (\ref{eq:def.Er.switched}). Observe that the set $E$ so defined satisfies the assumptions of Lemma~\ref{lem:I_rE.properties.switched} and hence there exists an integer $\bar{r}\geq 0$ such that $\mathcal{S}=I_r(E)$ for all $r\geq \bar{r}$. Consequently, the optimal value of the convex quadratic program in (\ref{eq:QP.switched}) with $H_{\pi}=H_{\pi,0}$ and $\bar{\alpha}=1$ is equal to $c^*$ for any $r \geq \bar{r}$. As this optimal value is an upperbound on the optimal value of the SDP in (\ref{eq:inner.ellipse.sdp.switched}) for any $r \geq 0$ (indeed, $H_{\pi,0}$ is always feasible to (\ref{eq:inner.ellipse.sdp.switched})), the claim follows. Finally, as any optimal solution $x_r$ to the SDP at level $r\geq \bar{r}$ satisfies $c^Tx_r=c^*$ and belongs to $\mathcal{S}$, it must be an optimal solution to the \aaa{R-UTVLD-LP} as well. 
\end{proof}


%
%


We end with a numerical example.

\begin{Example}\label{ex:switched}
Consider an \aaa{R-UTVLD-LP} defined by the following data:
\begin{align*}
A=\begin{pmatrix}1 & 0\\ -1.5 & 0\\ 0 & 1\\ 0 & -1\\ 1 & 1\end{pmatrix},
b=\begin{pmatrix} 1 \\ 1 \\ 1 \\ 1\\ 1\end{pmatrix},
c=\begin{pmatrix} 0.5 \\ 1 \end{pmatrix},
G_1=\alpha \begin{pmatrix} -1 & -1 \\ -4 & 0 \end{pmatrix}, \text{ and }
G_2=\alpha \begin{pmatrix} 3 & 3 \\ -2 & 1 \end{pmatrix},
\end{align*} 
with $\alpha=0.254$. For this value of $\alpha$ (and in fact for any $\alpha\geq 0.252$), there is no ellipsoid that is invariant under multiplication by the pair $\mathcal{G}=\{G_1,G_2\}$. This can be seen by observing that the SDP in (\ref{eq:converse.thm.LMIs}) is infeasible when $l=1$. However, feasibility of this SDP with $l=2$ shows that there are two ellipsoids whose intersection is invariant under the action of $G_1$ and $G_2$.\footnote{For $\alpha\geq 0.256$, we have $\rho(G_1,G_2)>1,$ and hence no compact full-dimensional set can be invariant under the action of $G_1$ and $G_2$. The fact that $\rho(G_1,G_2)>1$ can be seen by observing that $\sqrt{\rho(G_1G_2)}$ is a lower bound on $\rho(G_1,G_2)$~\cite{Raphael_Book} and that $\sqrt{\rho(G_1G_2)}=1.0029$ when $\alpha=0.256$.}

In Table \ref{tab:ulbs.RULDLP}, we give upper and lower bounds on the optimal value of this \aaa{R-UTVLD-LP}. To obtain the lower bounds, we minimize $c^Tx$ over the sets $S_r$ in (\ref{eq:Sr.switched}) for $r=0,1,2$. To obtain the upper bounds, we solve the SDP in (\ref{eq:inner.ellipse.sdp.switched}) for $l=2$ and $r=0,1,2$. For the convenience of the reader, we write out this SDP ($a_j^T$ here denotes the $j$-th row of the matrix $A$):
\begin{align}\label{eq:sdp.example.switched}
\underset{x\in\mathbb{R}^2, Q_{1,2}\in S^{2\times 2}}{\text{minimize}} \hspace*{1cm} & c^Tx  \\
\text{s.t. } \hspace*{1cm}& Q_1\succ 0, Q_2 \succ 0, \nonumber\\
& G_1Q_1G_1^T\preceq Q_1, ~ G_2Q_2G_2^T\preceq Q_1,~G_1Q_1G_1^T\preceq Q_2,~G_2Q_2G_2^T\preceq Q_2, \nonumber\\
& \begin{bmatrix} Q_i & Gx\\ (Gx)^T &1 \end{bmatrix} \succeq 0, ~\forall G \in \mathcal{G}^r,~i=1,2, \nonumber\\
& a_j^TQ_1a_j \leq 1,j=1,\ldots,5, \nonumber\\
& AGx \leq 1, ~\forall G \in \mathcal{G}^{k}, k=0,\ldots,r-1. \nonumber
\end{align}

From Table~\ref{tab:ulbs.RULDLP}, we note that as expected, our sequence of upper bounds (resp. lower bounds) are nonincreasing (resp. nondecreasing). Though we know that these bounds must converge to the optimal value of our \aaa{R-UTVLD-LP} in finite time, convergence has not occurred in this example in 3 iterations. Indeed, the gap between the upper bound and lower bound for $r=2$ is quite small but still nonzero.

%

\begin{table}[H]
	\begin{center}
		\begin{tabular}{ c|c|c|c } 
			& $r=0$ & $r=1$ & r=2 \\ 
			\hline
			Lower bounds obtained by minimizing $c^Tx$ over $S_r$ &  -1.3333   & -0.9374 &  -0.8657 \\ 
			Upper bounds obtained by minimizing $c^Tx$ over $I_r$ &
			
			  -0.7973  & -0.8249 &   -0.8417\\
		\end{tabular}
	\end{center}
	\caption{Our lower and upper bounds on the optimal value of the \aaa{R-UTVLD-LP} in Example~\ref{ex:switched}.}
	\label{tab:ulbs.RULDLP}
\end{table}

In Figure \ref{fig:RULDLPexp}, we have plotted the outer approximations $S_r$ to the set $\mathcal{S}$ in dark gray, and the inner approximations $I_r(E_r)$ to the set $\mathcal{S}$ in light gray. To be more specific, let $Q_{1,r}, Q_{2,r}$ be optimal matrices to the SDP in (\ref{eq:sdp.example.switched}) at level $r$. The sets $I_r(E_r)$ that are depicted are defined as:
\begin{equation}\nonumber
\begin{aligned}
I_r(E_r)=\Big\{z\in\mathbb{R}^n|~\ (Gz)^TQ_{1,r}^{-1}(Gz)\leq 1, (Gz)^TQ_{2,r}^{-1}(Gz)\leq 1, \forall G \in \mathcal{G}^r, AG z\leq 1, \forall G \in \bigcup_{k=0}^{r-1} \mathcal{G}^k\Big\}.
\end{aligned}
\end{equation}

%
 
 In each subfigure, we have also plotted the optimal solutions achieved by minimizing $c^Tx$ over the inner and outer approximations to $\mathcal{S}$. Note that as $r$ increases, the set $\mathcal{S}$ gets sandwiched between these two approximations more and more tightly.

\begin{figure}[H]
	\centering
	\begin{subfigure}{.32\textwidth}
		\centering
		\includegraphics[scale=0.19]{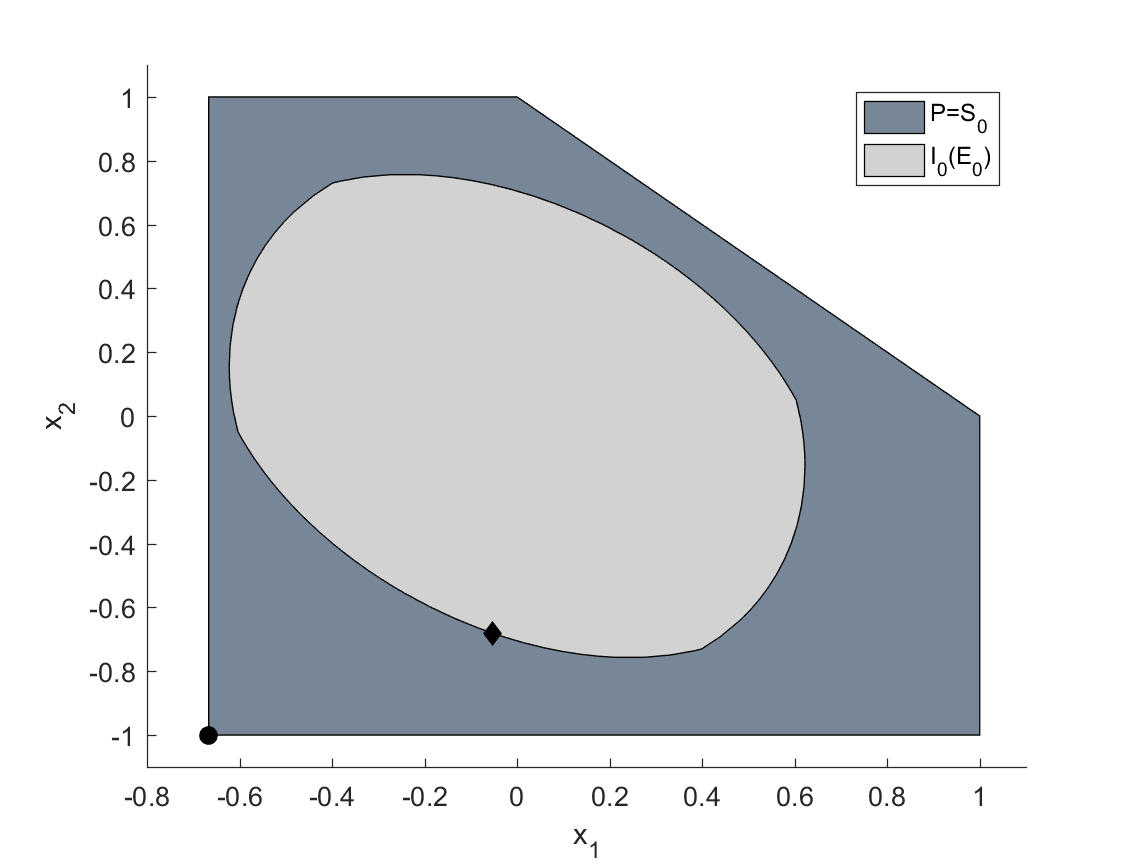}
		\caption{$r=0$}
		\label{fig:RULDLPr0}
	\end{subfigure}%
	\begin{subfigure}{.32\textwidth}
		\centering
		\includegraphics[scale=0.19]{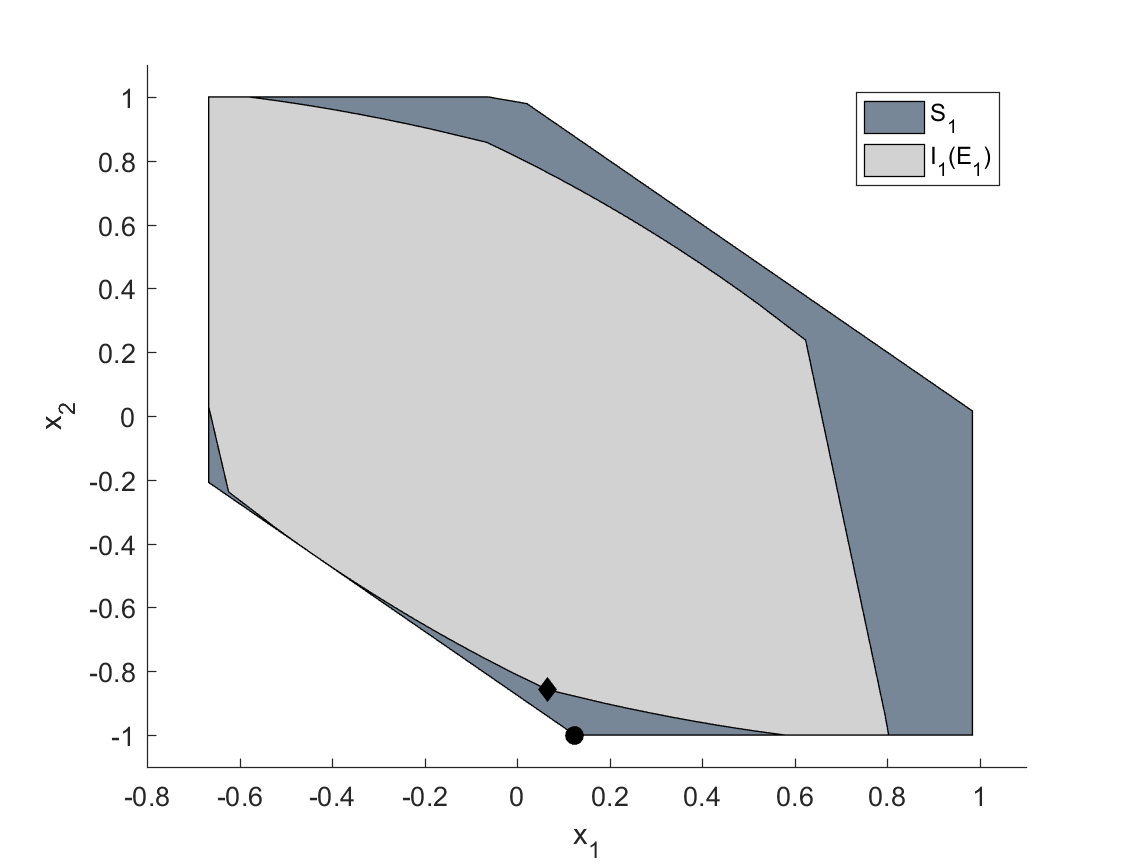}
		\caption{$r=1$}
		\label{fig:RULDLPr1}
	\end{subfigure}
\begin{subfigure}{.32\textwidth}
	\centering
	\includegraphics[scale=0.19]{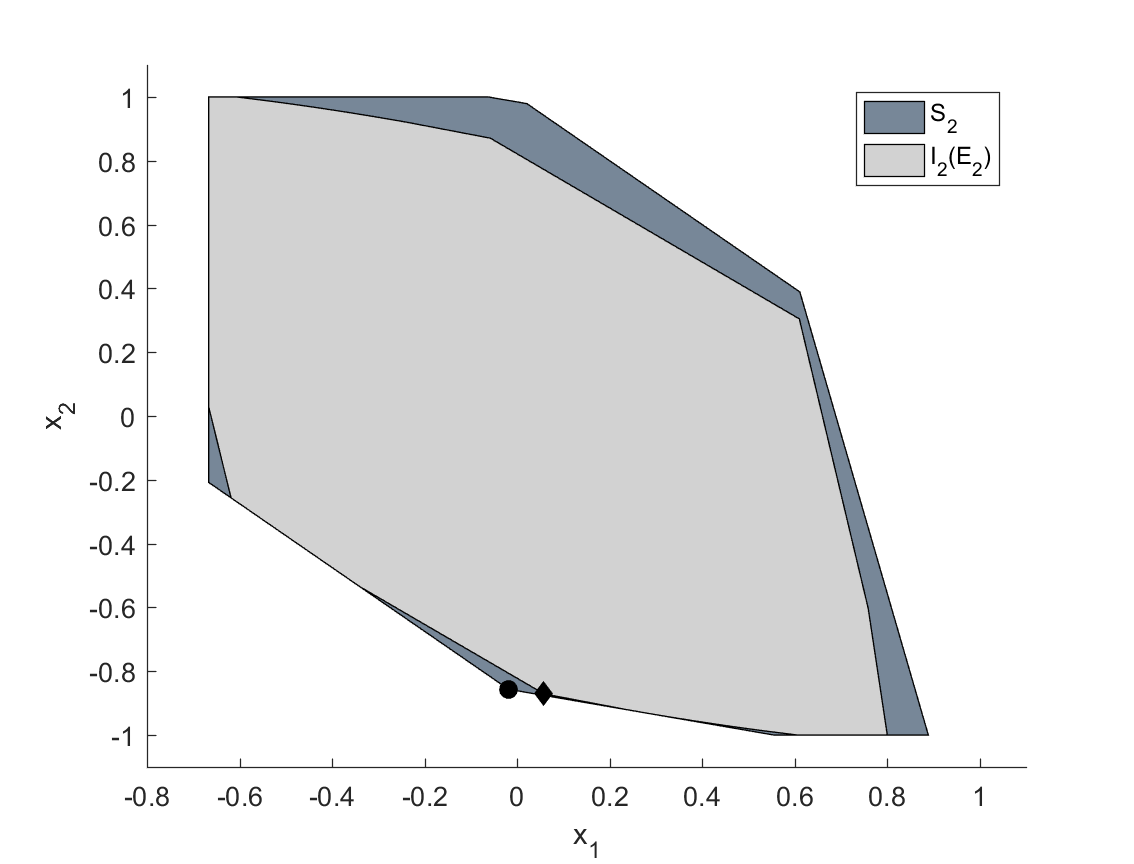}
	\caption{$r=2$}
	\label{fig:RULDLPr2}
\end{subfigure}
	\caption{Three levels of our inner and outer approximations to the feasible set of the \aaa{R-UTVLD-LP} in Example~\ref{ex:switched}.}
	\label{fig:RULDLPexp}
\end{figure}

\end{Example}

\subsection{\aaa{ A Special Case: Permutation Matrices}}\label{sec:permutation}
\aaa{In this section, we consider a special case of R-UTVLD-LP when the input matrices governing the dynamical system in~\eqref{eq:difference.inclusion} are given by $n!$ permutation matrices:
$$\mathcal{P}=\big\{\Pi\in\R^{n\times n} \:|\: \Pi\text{ is a permutation matrix}\big\}.$$
Recall that a {\em permutation matrix} is a square binary matrix with each row and each column containing exactly one nonzero entry. The joint spectral radius of the set of permutation matrices is exactly 1. 

A useful property of $\mathcal{P}$ is that its elements form a group under matrix multiplication. Consequently, we have
\begin{equation}\label{eq:Permutation.product}
\Pi^1,\Pi^2\in \mathcal{P}\quad\implies\quad \Pi^1 \Pi^2\in \mathcal{P},    
\end{equation}
and therefore all finite products of matrices from $\mathcal{P}$ belong to $\mathcal{P}$. 
Moreover, the convex hull of $\mathcal{P}$ has a simple linear description \cite{Birkhoff46}:
\begin{equation}\label{eq:assignment}
	\text{conv}(\mathcal{P})~=~\Big\{\Pi\in\R^{n\times n} 
	\:|\: \sum_{j\in N} \Pi_{ij}=1~\forall i\in N,~\sum_{i\in N} \Pi_{ij}=1~\forall j\in N,~\Pi_{ij}\ge0~\forall i,j\in N\Big\}
\end{equation}
where, $N=\{1,\ldots,n\}$.
When  the linear dynamics of an \aaa{R-UTVLD-LP} is given by permutation matrices, using the same notation as in \eqref{eq:RULDLP} and the implication in~\eqref{eq:Permutation.product}, the feasible set of the \aaa{R-UTVLD-LP} can be written as
\begin{align*}
	\mathcal{S}\mathrel{\mathop{:}}
	&= \bigcap_{k=0}^{\infty} \big\{x \in \mathbb{R}^n| ~A\Pi x \leq b, \forall \Pi\in\mathcal{P}^k\big\}
		=  \big\{x \in \mathbb{R}^n| ~A\Pi x \leq b, \forall \Pi\in\mathcal{P}\big\}.
\end{align*}
We next show that this feasible set has a polynomial-size polyhedral description in extended space (and hence \aaa{R-UTVLD-LP} can be reformulated as polynomial-size linear program in this case).

\begin{theorem}\label{thm:birkoff} Consider the \aaa{R-UTVLD-LP} in \eqref{eq:RULDLP} with $P= \{x\in\R^n|\ Ax\leq b  \},$ where $A\in~\R^{m\times n}$, and $\mathcal{G}= \mathcal{P}$. Then, the feasible set $\mathcal{S}$ of the \aaa{R-UTVLD-LP} can be written as
	$$ \mathcal{S}~=~\Big\{x \in \mathbb{R}^n
	~|~\sum_{i\in N}u_{i}^k+\sum_{j\in N}w_{j}^k\le b_k~\forall k\in\{1,\ldots,m\}
	,~u_{i}^k+w_{j}^k\ge A_{ki}x_j~\forall i,j\in N,k\in\{1,\ldots,m\}\Big\}.$$
\end{theorem}
\begin{proof}
First note that $\mathcal{S} = \cap_{k=1}^m \mathcal{S}_k,$ where
$$ \mathcal{S}_k~=~\big\{x \in \mathbb{R}^n~| ~a_k^T\Pi x \leq b_k, \forall \Pi\in\mathcal{P}\big\}.$$
Here, $a_k^T$ denotes the $k$th row of $A$. 
For a given point $\bar x\in\R^n$, we have $\bar x\in\mathcal{S}_k$ if and only if the optimal value $z_k$ to the linear program
$$\max\big\{a_k^T\Pi \bar x ~|~  \Pi\in\mathcal{P}\big\}=\max\big\{a_k^T\Pi \bar x ~|~  \Pi\in\text{conv}(\mathcal{P})\big\}$$
does not exceed $b_k$. Using LP duality and \eqref{eq:assignment}, we have
\begin{align*}z_k
 &=\max\big\{\sum_{i,j\in N}  A_{ki}\bar x_j\Pi_{ij}~|~ \sum_{j\in N} \Pi_{ij}=1~\forall i\in N,~\sum_{i\in N} \Pi_{ij}=1~\forall j\in N,~\Pi_{ij}\ge0~\forall i,j\in N\big\}\\
 &=\min\big\{\sum_{i\in N} u_{i}^{k}+\sum_{j\in N} w_{j}^{k}~|~ 
 u_{i}^{k}+w_{j}^{k}\ge A_{ki}\bar x_j~\forall i,j\in N\big\},
 \end{align*}
where $u^{k},w^{k}\in\R^n$ are the dual variables associated with the constraints in the maximization problem.
As both problem are feasible, this implies that $\bar x\in\mathcal{S}_k$ if and only if the minimization problem above has a solution $\bar u^{k}$ and $\bar w^{k}$ with  
$$\sum_{i\in N}\bar u_{i}^k+\sum_{j\in N}\bar w_{j}^k\le b_k.$$
Consequently, we have
$$ \mathcal{S}_k~=~\Big\{x \in \mathbb{R}^n
~|~\sum_{i\in N}u_{i}^k+\sum_{j\in N}w_{j}^k\le b_k,~u_{i}^k+w_{j}^k\ge A_{ki}x_j~\forall i,j\in N,~u^{k},w^{ k}\in\R^n\Big\}$$
and as $\mathcal{S} = \cap_{k=1}^m \mathcal{S}_k$, the proof is complete.
\end{proof}
We remark that more generally, whenever one has a R-UTVLD-LP problem with dynamics matrices $\mathcal{G}$ in such a way that a polynomial-size LP (or SDP) based description of $\text{conv}({\mathcal{G}^*})$ is available, then one can use the duality arguments above to reformulate the problem as a polynomial-size LP (or SDP).
}

\section{Future directions and a broader agenda: optimization with dynamical systems constraints}\label{sec:opt.with.DS}

In this paper, we studied robust-to-dynamics optimization (RDO) problems where the optimization problem is an LP and \aaar{the dynamics is governed either by a known, or an unknown and time-varying linear system.} Even in these two settings, a number of questions remain open. \aaar{For example, is the problem of testing membership of a point to the feasible set of an R-LD-LP decidable? 
We have shown that this problem is NP-hard in general and polynomial-time solvable when the spectral radius of the linear map is bounded away from one.
It follows from~\cite{almagor2019a} that the general problem is decidable when  $n\leq 3$. In higher dimensions, however, decidability is unknown even in the special case where the input polyhedron is a single halfspace. This special case is related to the so-called Skolem-Pisot problem; see, e.g.,~\cite{Karimov2022,Karimov2023} and references therein.
} In the context of \aaa{R-UTVLD-LPs}, when $\rho(\mathcal{G})<1$ and the polytope $P$ containing the origin in its interior, can one analyze the number of steps needed for our inner and outer approximations to coincide as we did for the case of R-LD-LPs?


\aaar{The algorithmic analysis of RDO problems involving other types of optimization problems and dynamical systems is also left for future research. We note that in full generality, basic questions around RDO can be undecidable. For instance, it follows from~\cite[Corollary 1]{halava2007improved} that when the dynamics is as in Section~\ref{sec:R-ULD-LP} with number of matrices $s$ equal to 2 and the dimension $n$ equal to 9, and when the input set $\Omega$ is the complement of a hyperplane, then testing membership of a given point to the feasible set $\mathcal{S}$ of the RDO is undecidable. Similarly, one can observe that the same question is undecidable  when the dynamical systems is given by degree-2 polynomial differential equations and the input set $\Omega$ is an open halfspace; see~\cite[Theorem 26]{hainry2009}, and also~\cite{Andrade2023} for a similar result regarding dynamical systems arising from no-regret learning in games. Despite these negative results, we believe that future research can identify additional special cases of RDO problems that admit tractable algorithms. 


}


More generally, we believe that optimization problems that incorporate \emph{``dynamical systems (DS) constraints''} in addition to standard mathematical programming constraints can be of interest to the optimization community at large. An optimization problem with DS constraints is a problem of the form:

\begin{equation}\label{eq:opt.dynamic}
\begin{array}{lll}
\mbox{minimize} &  f(x) &\ \\ 
\mbox{subject to} & x\in \Omega\cap \Omega_{DS}. &\ 
\end{array}
\end{equation}
Here, the set $\Omega$ is the feasible set of a standard mathematical program and is described in an explicit functional form; i.e., $\Omega\mathrel{\mathop:}=\{x\in\mathbb{R}^n|\  h_i(x)\leq 0, i=1,\ldots,m \}$ for some scalar valued functions $h_i$. The constraint set $\Omega_{DS}$ however is defined implicitly and always in relation to a dynamical system given in explicit form $$\dot{x}=g(x) \ \mbox{(in continuous time)} \quad \mbox{or} \quad  x_{k+1}=g(x_k) \ \mbox{(in discrete time)}.$$ The set $\Omega_{DS}$ corresponds to points whose future trajectory satisfies a \emph{prespecified desired property over time}. The optimization problem (\ref{eq:opt.dynamic}) with different DS constraints can look like any of the following:


Optimize $f$ over the points in $\Omega$ whose future trajectories under the dynamical system $g$
\begin{itemize}
\item stay in $\Omega$ for all time {\it (invariance)},
\item asymptotically flow to the origin {\it (asymptotic stability)},
\item never intersect a given set $\Theta$ {\it (collision avoidance)},
\item reach a given set $\Theta$ in finite time {\it (reachability)}, etc.
\end{itemize}

\aaa{One can imagine these problems arising from mathematical models in several domains. For example, in robotics, the version with asymptotic stability could correspond to establishing how high a humanoid robot can raise one of its legs before losing balance; see Figure 6 in~\cite{majumdar2014control}. The version with collision avoidance could correspond to the problem of computing the highest initial speed at which an autonomous vehicle can perform a turn without going off track.}

Figure~\ref{fig:opt.dynamics.constraints} gives an example of a two-dimensional optimization problem with DS constraints. Here, the objective function is $f(x)=-x_2$, the set $\Omega$ (plotted in blue) is defined by a linear inequality and a convex quadratic inequality, and the dynamical system is a cubic differential equation, $g(x)=(-x_2+3x_1x_2,x_1-\frac{1}{2}x_1^2x_2)^T,$ whose resulting vector field is plotted with little orange arrows. A DS constraint can be any of the four items listed above with $\Theta$ being the red triangle. 

%
%
%
%
%
%
%
%


%
%
%
%

\begin{figure}[h]
\centering
    \includegraphics[scale=0.25]{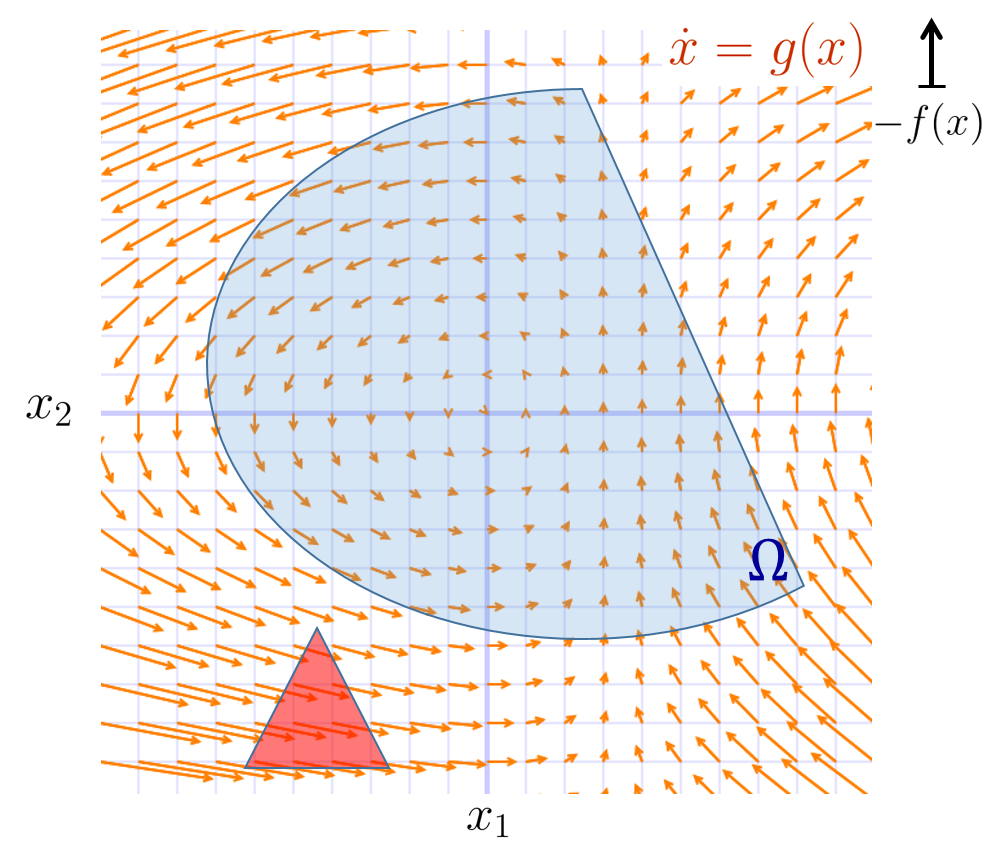}
    \caption{An illustration of an optimization problem with dynamical systems (DS) constraints.}
     \label{fig:opt.dynamics.constraints}
\end{figure}

We would cautiously argue that the optimization community is by and large not used to thinking about constraints that are defined implicitly via a dynamical system. At the other end of the spectrum, while the study of invariant sets, regions of attraction, reachable sets, etc. is very natural for control theorists, their focus often is not on settings where one needs to optimize over these sets subject to a melding with mathematical programming constraints. 

An interesting future research agenda would be to initiate a systematic algorithmic study of optimization problems with DS constraints. By this, we mean a rigorous complexity study of problems of this type, via, e.g., either a polynomial-time algorithm, or an NP-hardness/undecidability result, or an approximation algorithm. As can be seen from the table below, this paper has only covered the tip of the iceberg when it comes to such problems. Indeed, a class of optimization problems with DS constraints is defined by taking one element of each of the three columns of this table. The starred entries correspond to cases to which this paper has contributed partial results.

%
%
%
%
%
%
%
%
%
%
%
%
%


%

\begin{center}

\begin{table}[H]
\small
\centering
\begin{tabular}{l|l|l}
\textbf{Opt. Problem ``$f,\Omega$''} & \textbf{Type of Dynamical System ``$g$''}       & \textbf{DS Constraint ``$\Omega_{DS}$''}          \\ \hline
Linear program*                & Linear*                  &Invariance*          \\
Convex quadratic program               & Linear and uncertain/stochastic                 &Inclusion in region of attraction      \\
Semidefinite program          & Linear, uncertain, and time-varying*                   &Collision avoidance      \\
Robust linear program             & Nonlinear (e.g., polynomial)                  &Reachability    \\
0/1 integer program       & Nonlinear and time-varying & Orbital stability\\
      Polynomial program    & Discrete/continuous/hybrid of both  & Stochastic stability\\
\vdots         & \vdots  &\vdots
\end{tabular}
\end{table}
\end{center}

\aaa{\st{We believe that optimization problems with dynamical systems constraints are fundamental enough to find applications in different areas, and that their study will surely lead to interesting mathematics at the intersection of algebra and geometry.}}

%
%
%
%


\section{Acknowledgments} The authors are grateful to Georgina Hall for several insightful discussions (particularly around the content of Proposition~\ref{prop:P.unbounded.not.finite} and Section~\ref{subsubsec:IrEr.switched}), and to two anonymous referees whose comments have grately improved our manuscript.



\bibliographystyle{siamplain}
\bibliography{pablo_amirali}

\end{document}